\documentclass[a4paper,14pt]{article}

\usepackage[T2A]{fontenc}
\usepackage[english]{babel}
\usepackage{amsmath}
\usepackage{amsfonts}
\usepackage{amssymb}

\usepackage[a4paper,includeheadfoot,nohead,mag=1000, left=2.5cm,right=1.5cm,top=1cm,bottom=1cm]{geometry}
\usepackage{amsthm} 
\usepackage{tikz}
\usepackage{tikz-cd}
\usepackage{cite}
\usepackage{comment}
\usepackage{import}
\usepackage{setspace}
\usepackage{xfrac}    
\usepackage{mathrsfs}

\theoremstyle{plain}
\newtheorem{theorem}{Theorem}[section]
\newtheorem{lemma}[theorem]{Lemma}
\newtheorem{proposition}[theorem]{Proposition}
\newtheorem{corollary}[theorem]{Corollary}
\theoremstyle{remark}
\newtheorem{remark}{Remark}[section]\theoremstyle{definition}
\newtheorem{definition}{Definition}[section]

\begin{document}

\frenchspacing

\title{Functional approach to the normality of mappings}

\author{M.\,\,Yu.~Liseev}

\maketitle

\begin{abstract}
	In the article a technique of the usage of $f$-continuous functions (on mappings) and their families is developed. 
A  proof of the Urysohn's Lemma for mappings is presented and a variant of the Brouwer-Tietze-Urysohn Extension Theorem for mappings is proven.  Characterizations of the normality properties of mappings are given and the notion of a perfect normality of a mapping is introduced. It seems to be the most optimal in this approach.

Bibliography: 6 names.

\end{abstract}

\footnotetext{«The paper was published with the financial support of the Ministry of Education and Science of the Russian Federation as part of the program of the Moscow Center for Fundamental and Applied Mathematics under the agreement №075-15-2019-1621.}

\section{Introduction and preliminary information}
The fiberwise general topology (also called the topology of continuous maps) is developing on the basis of general and algebraic topologies. The idea of extending topological properties of spaces to mappings ("from spaces --- to mappings") was formulated by B.\,A.~Pasinkov in~\cite{PasOtobrFunct} and inspired numerous studies.
The article provides a functional approach to the solution of the problem  of extending normality properties to mappings which was proposed to the author by B.\,A.~Pasynkov. 
In particular, the extension to mappings of the Brouwer-Tietze-Urysohn Theorem about the extension
of functions from closed subsets of normal spaces onto the whole space is considered. The concept of a perfectly normal mapping is introduced, for which the  analogue of Vedenisov's condition for perfectly normal spaces is fulfilled.

Definition of a normal mapping~\cite{MusPas} has led to the natural definition of a $co$-perfectly normal mapping (a normal mapping, every open submapping of which is a $F_{\sigma}$-submapping).
In~\cite{liseev3}, examples of $co$-perfectly normal mappings that are not hereditarily normal are given, the question of inheriting normality of a mapping by $F_{\sigma}$-submappings remains open.
The "enhanced" property of the mapping normality --- $\sigma$--normality was also introduced in~\cite{liseev3}. Every $\sigma$--normal mapping is normal, and for a constant $T_{1}$--mapping its normality, $\sigma$--normality and normality of a total space are equivalent properties~\cite{liseev3}.
As it turned out, the  mapping $\sigma$--normality is inherited by its $F_{\sigma}$-submappings.
Thus, this approach made it possible to introduce in~\cite{liseev3} the concept of a $co$-$\sigma$--perfect normality of a mapping (i.e. a $\sigma$--normal mapping, every open subset of which has the type $F_{\sigma}$) and prove that, defined in this way, the perfect normality of a mapping implies its hereditary normality, and is a hereditary property.

The concept of a $f$--continuous function on a mapping $f:X \to Y$ was first introduced by A.\,Yu.~Zubov~\cite[Theorem 3.1]{Zubov}. Using this concept, he proved an analogue of Uryson's Lemma (formulated in~\cite[Theorem 3.1]{Zubov}) for mappings. It seems that no further progress in the functional approach to study mappings (using $f$-continuous functions) occurred.

The designations and the main notions used in the paper are given in~\S\ 1. In~\S\ 2 
a technique of the usage of $f$--continuous functions is developed, and explicit methods for constructing such functions are given (Lemmas~\ref{lemFContFunc}, \ref{lemSumFcontMapFcont} and Proposition~\ref{prFContMapOnCorPart}).

The main results are contained in~\S\S\ 3 --- 5. Theorem~\ref{thMainFuncChar} of~\S\ 3 gives characterizations of a mapping normality using analogues of the Urysohn's Lemmas and the Brouwer--Tietze--Urysohn Extension Theorem for mappings.
In its proof a $f$-continuous function separating disjunct closed subsets is explicitly constructed.
In the given construction a $f$--continuous function which is continuous on the "limit fiber" of the mapping, is obtaned by approximating it by stepwise functions. 
The condition binding the characterization of normality is a generalized analogue of small Urysohn's Lemma (Lemma~\ref{lemFSetPartitionofNormalMap}).

In~\S\ 4 the concept of a family of functions $f$-equicontinuous at a point is introduced (Definition~\ref{defEquicontMapSeq}). This concept made it possible to characterize the $\sigma$-normality of a mapping (Theorem~\ref{thFuncCharSigmNorm}).

In~\S\ 5 the concept of a mapping perfect normality is introduced. It uses the family of functions $f$-equicontinuous at a point and extends Vedenisov's perfect normality condition for spaces to mappings (Definition~\ref{defCommonPerfNorm}). The introduced concept of perfect normality is a hereditary property (Proposition~\ref{prHerFuncPerfNormMap}) and a perfectly normal mapping is hereditarily normal (Theorem~\ref{thCommonPerfNormMapIsHerNorm}). Proposition~\ref{prHerFuncPerfNormMapbb} shows that this definition of a mapping perfect normality is, in fact, intermediate, between the previously introduced concepts of $co$-$\sigma$--perfect normality and $co$--perfect normality, and appears to be optimal.

\bigskip


All spaces are topological spaces, $\tau_{X}$ is the topology of the space $X$,  $\mathcal{N}(x)$ is the family of open neighborhoods of a point $x \in X$, mapping $f: X\to Y$ is continuous. The first countable ordinal is $\omega_{0} = \{ 0 \} \cup \mathbb{N} $. The sequence of non-negative integers $a, a+1, \ldots, b-1, b $ is denoted by $\overline{a,b}$.

Real valued mappings (not necessarily continuous) $\varphi: X\to\mathbb R$ are called functions. The {\it norm} of a bounded function $\varphi : X \to \mathbb{R}$ is $||\varphi||= \sup\limits_{x \in X} |\varphi(x)|$, its 
{\it oscillation at a point $x \in X$} is 
\begin{equation*}
\mathrm{osc}_{\varphi}(x) = \inf \{ \sup\limits_{y \in U} |\varphi(x)-\varphi(y)|: U \in \mathcal{N} (x)\}\geqslant 0.
\end{equation*}

For a non-empty subset $A \subset X$ we set $\mathrm{osc}_{\varphi}(A) = \mathrm{sup} \{ \mathrm{osc}_{\varphi}(x): x \in A \}$ and $\mathrm{osc}(\varnothing) = 0$. 

\begin{remark}\label{remOsc}
It is easy to check the following.

$\mathrm{osc}_{\varphi}(x)\leqslant 2 \cdot ||\varphi||$ for any point $x \in X$, and the function $\varphi$ is continuous at $x$ iff $\mathrm{osc}_{\varphi}(x)=0$.

For any bounded functions $\varphi : X \to\mathbb{R}$, $\psi : X \to\mathbb{R}$, and any numbers $\alpha, \beta\in\mathbb {R}$ the following  inequality holds 
\begin{equation*}
\mathrm{osc}_{\alpha\varphi+\beta\psi}(A)\leqslant |\alpha|\cdot\mathrm{osc}_{\varphi}(A)+|\beta|\cdot\mathrm{ osc}_{\psi}(A).
\end{equation*}
\end{remark}

\begin{lemma} \label{lemOscDisj} For a bounded function $\varphi : X \to\mathbb{R}$ and numbers $ a < b $ such that $b - a > \mathrm{osc}_{\varphi} (X)$ the following hold
\begin{equation}
	\begin{aligned}
	\{x\in X\ |\ \varphi(x)\leqslant a \} \cap \mathrm{cl} \{x\in X\ |\ \varphi(x)\geqslant b\} = \varnothing, \ \
	\mathrm{cl} \{x\in X\ |\ \varphi(x)\leqslant a\}\cap\{x\in X\ |\ \varphi(x)\geqslant b\} = \varnothing.
	\end{aligned}
	\label{eqnoO1}
	\end{equation}
\end{lemma}
	
\begin{proof} Let the sets $A=\{x\in X\ |\ \varphi(x)\leqslant a\}$ and $B=\{x\in X\ |\ \varphi(x) \geqslant b\}$ be non-empty.
	
For any point $x \in A$ there is its open neighborhood $\mathcal{O}x$ such that
	\begin{equation*}
		 |\varphi(x) - \varphi(z)| < b - a \text{ for any point } z\in \mathcal{O}x.
	\end{equation*}
Then $A\subset \bigcup\{\mathcal{O}x\ |\ \varphi(x)\leqslant a\} = \mathcal{O} $ and $ \varphi(z) < b$ for any point $ z \in \mathcal{O} $.
The set $\mathcal{O}$ is open and $\mathcal{O} \cap B= \varnothing$. Thus, the first formula in~(\ref{eqnoO1}) is proved. The second formula in~(\ref{eqnoO1}) can be proved similarly.
\end{proof}


\section{$f$--continuity of a function at a point}

\begin{definition}\cite[p. 111, the definition of a $\mathcal{G}_{y}(f)$-finally continuous function]{Zubov}
For a mapping $f: X \to Y$, a bounded function $\varphi: X \to\mathbb{R}$ is called $f$--{\it continuous at a point} $y \in Y$ 

if for any $ \varepsilon > 0$ there is a neighborhood $\mathcal{O}y$ of the point $y$ such that $\mathrm{osc}_{\varphi}(f^{-1}\mathcal{O}y) < \varepsilon $.
\label{defFcontMap}
\end{definition}

\begin{remark} \label{remFcontMap} From the Definition~{\rm\ref{defFcontMap}} it follows that a $f$--continuous at a point $y$ function $\varphi$ is continuous at every point $x \in f^ {-1}G$ of some $G_{\delta}$--subset $G \subset Y$, $y \in G$.
In particular, at every point of the fiber $f^{-1}y$.

Any bounded function $\varphi: X \to\mathbb{R}$ is $f$-continuous at the point $y \in Y$ if $y \not\in \mathrm{cl} \big( f(X) \big)$.

Linear combination $\alpha\varphi+\beta\psi$ of $f$-continuous at the point $y \in Y$ functions $\varphi : X \to\mathbb{R}$, $\psi : X \to\mathbb {R}$, $\alpha, \beta\in\mathbb{R}$, is a $f$-continuous at the point $y \in Y$ function. $f$-continuous at a point $y$ functions form a linear subspace of functions.
\end{remark}

\begin{lemma} \label{lemFContFunc} Let for a mapping $f: X \to Y$ and a point $y\in Y$ the pairs $(\mathcal{O}_{n}, \varphi_{n})$, where $\mathcal{O}_{n}$ is a neighborhood of $y$, $\varphi_{n}: X\to [0, 1]$, $n \in \omega_{0}$,  are such that
\begin{itemize}
	\item[\rm(a)] $\mathcal{O}_{n+1}\subset \mathcal{O}_{n}$, $n\in\omega_0$, $\mathcal{O}_{ 0} = Y$;
	\item[\rm(b)] the sequence $\{ \mathrm{osc}_{\varphi_{n}}(f^{-1}\mathcal{O}_{n}) \}_{n \in \mathbb{N}}$ converges to $0$;
	\item[\rm(c)] the number series $\sum\limits_{n=0}^{\infty}||\varphi_{n+1}|_{f^{-1}\mathcal{O}_{ n+1}} - \varphi_{n}|_{f^{-1}\mathcal{O}_{n+1}}||$ converges.
\end{itemize}
Then the function $\varphi: X \to [0, 1]$
\begin{equation*}
	\varphi (x) = \left\{ \begin{array}{ll}
	\varphi_{n}(x), \ & x \in f^{-1}\mathcal{O}_{n} \setminus f^{-1}\mathcal{O}_{n+1}, n \in \omega_{0},\\
	\lim\limits_{n \to \infty}\varphi_{n}(x), \ & x\in\bigcap\limits_{n=0}^{\infty} f^{-1}\mathcal{O} _{n}\\
	\end{array}\right.
\end{equation*}
is $f$-continuous at $y \in Y$.
\end{lemma}
	
\begin{proof} For any $n, m \in \mathbb{N}$ the following inequality holds
	\begin{equation}
	\begin{aligned}
	||\varphi_{n+m}|_{f^{-1}\mathcal{O}_{n+m}} - \varphi_{n}|_{f^{-1}\mathcal{O} _{n+m}}|| \leqslant ||\varphi_{n+m}|_{f^{-1}\mathcal{O}_{n+m}}-\varphi_{n+m-1}|_{f^{-1 }\mathcal{O}_{n+m}}|| +\ldots\\
	\ldots +||\varphi_{n+1}|_{f^{-1}\mathcal{O}_{n+1}}-\varphi_{n}|_{f^{-1}\mathcal {O}_{n+1}}||. \label{eqno(1.1)}
	\end{aligned}
	\end{equation}
It follows from condition (c) and inequality (\ref{eqno(1.1)}) that the functional sequence $\{ \varphi_{n}|_{f^{-1}(\bigcap\limits_{n=0}^ {\infty} \mathcal{O}_{n})}\}_{n \in \mathbb{N}}$ converges uniformly on $f^{-1}(\bigcap\limits_{n=0}^ {\infty}\mathcal{O}_{n})$. Thus, the function $\varphi: X\to [0, 1]$ is well defined.
	
Let $\varepsilon > 0$ be arbitrary. By condition (b) there exists $N_{1} \in \mathbb{N}$ such that for all $n \geqslant N_{1}$, $\mathrm{osc}_{\varphi_{n}}(f ^{-1}\mathcal{O}_{n})<\frac{\varepsilon}{3}$. By condition (c), there exists $N_{2}\in \mathbb{N}$ such that $\sum\limits_{n=N_{2}}^{\infty}||\varphi_{n+1}| _{f^{-1}\mathcal{O}_{n+1}}-\varphi_{n}|_{f^{-1}\mathcal{O}_{n+1}}|| < \frac{\varepsilon}{3}$. Consider $N=\max\{N_1, N_2\}$. For any point $ x \in f^{-1}\mathcal{O}_{N}$ there exists a neighborhood $ U \in \mathcal{N }(x)$, $U\subset f^{-1}\mathcal{O}_{N}$, such that for any $\widetilde{x} \in U$ the following holds 
	\begin{equation} \label{eqno(1.11)}
	|\varphi_{N}(x)-\varphi_{N}(\widetilde{x})|<\frac{\varepsilon}{3}.
	\end{equation}
	
Since $\varphi (x) = \varphi_{N+m}(x)$, for $ x \in f^{-1}\mathcal{O}_{N + m} \setminus f^{-1 }\mathcal{O}_{N + m +1}$, or $\varphi (x) = \lim\limits_{n \to \infty}\varphi_{n}(x)$ otherwise (for $ \varphi (\widetilde{x})$ similarly), then from the condition $N \geqslant N_{2}$ and the inequality (\ref{eqno(1.1)}) it follows that
	\begin{equation}
	|\varphi (x)-\varphi_{N}(x)| \leqslant \frac{\varepsilon}{3} \text{ and } |\varphi (\widetilde{x})-\varphi_{N} (\widetilde{x})| \leqslant \frac{\varepsilon}{3}
	\label{eqno(1.111)}
	\end{equation}
	
Hence for any point $ \widetilde{x} \in U$ the following inequality holds 
	\begin{equation*}
	|\varphi (x)-\varphi (\widetilde{x})| \leqslant |\varphi (x)-\varphi_{N}(x)|+|\varphi_{N}(x)-\varphi_{N}(\widetilde{x})| + |\varphi_{N}(\widetilde{x}) - \varphi (\widetilde{x})|.
	\end{equation*}
Further, from (\ref{eqno(1.11)}) and (\ref{eqno(1.111)}) it follows that
	\begin{equation*}
	|\varphi (x)-\varphi (\widetilde{x})| < \frac{\varepsilon}{3} + \frac{\varepsilon}{3} + \frac{\varepsilon}{3} = \varepsilon.
	\end{equation*}
Therefore, $\mathrm{osc}_{\varphi}(f^{-1}\mathcal{O}_{N}) \leqslant \varepsilon$, and $f$-continuity of  $\varphi$ at $y$ is proved.
\end{proof}

\begin{definition} \label{defCorKfraction}
	{\rm (a)} For $k \in \mathbb{N}$ the partition $U^{0},\ldots, U^{k-1}$ of the space $X$ is called a {\it regular $k$--partition} if
	\begin{itemize}
	\item[{\rm (1)}] the set $\bigcup\limits_{m=0}^{p} U^{m}$ is closed in $X$ for $p\in\overline{0, k -1}$
	
	\noindent {\rm (}equivalently the set $\bigcup\limits_{m=l}^{k-1} U^{m}$ is open in $X$ for $l\in\overline{0, k- 1}${\rm )};
	\item[{\rm (2)}] $(\bigcup\limits_{m=0}^{p} U^{m}) \cap \mathrm{cl}(\bigcup\limits_{m=p+2 }^{k-1} U^{m}) = \varnothing,\ p\in\overline{0, k-3}$, $k \geqslant 3$.
	\end{itemize}
\end{definition}
	
\begin{lemma} For a regular $k$-partition, $k \geqslant 3$, of the space $X$, the following holds
	
	{\rm ($\ast$) $\bigcup\limits_{m=0}^{k-2}\mathrm{int} (U^{m} \cup U^{m+1}) = X$. }
	\label{remCorrPart}
\end{lemma}
	
\begin{proof}
	Let a regular $k$-partition $\{ U^{m} \}^{k-1}_{m=0}$ of $X$ be given.
	For each $i \in \overline{0,k-3}$ let us show that $U^{i} \subset \mathrm{int}(U^{i} \cup U^{i+1})$.
	By condition (2) of Definition~\ref{defCorKfraction} we have $\bigcup\limits_{m=0}^{i}U^{m} \ \cap \mathrm{cl} \bigcup\limits_{m=i+2 }^{k-1}U^{m} = \varnothing$.
	Then, $U^{i} \subset X \setminus \mathrm{cl} \bigcup\limits_{m=i+2}^{k-1}U^{m} = \mathrm{int}\Big( X \setminus \mathrm{cl} \bigcup\limits_{m=i+2}^{k-1}U^{m}\Big) \subset \mathrm{int}\Big( X \setminus \bigcup\limits_{ m=i+2}^{k-1}U^{m}\Big) = \mathrm{int} \Big( \bigcup\limits_{m=0}^{i+1}U^{m} \Big)$.
	Since $U^{i} \subset \bigcup\limits_{m=i}^{k-1}U^{m} $, then from condition (1) of Definition~\ref{defCorKfraction} we obtain that
	$$U^{i} \subset \mathrm{int} \Big( \bigcup\limits_{m=0}^{i+1}U^{m} \Big) \cap \bigcup\limits_{m=i }^{k-1}U^{m} = \mathrm{int} \Big( \bigcup\limits_{m=0}^{i+1}U^{m} \cap \bigcup\limits_{m=i}^{k-1}U^{m} \Big) = \mathrm{int}(U^{i} \cup U^{i+1}).$$
	By condition (1) of Definition~\ref{defCorKfraction} the union $U^{k-2} \cup U^{k-1} $ is open and $U^{k-2}, U^{k-1} \subset \mathrm{int}(U^{k-2} \cup U^{k-1})$. Finally, we obtain
	
	\begin{equation*}
	X = \bigcup\limits_{m=0}^{k-1} U^{m} \subset \bigcup\limits_{m=0}^{k-2}\mathrm{int} (U^{m} \cup U^{m+1}).
	\end{equation*}
\end{proof}

\begin{definition} \label{defCoherenceSetBinPart}
For a mapping $f: X\to Y$ and a point $y\in Y$,
the sequence $\{\mathcal{O}_{n}\}_{n\in \omega_{0}}$, $\mathcal{O}_{0}=Y$, $\mathcal{O}_{n+1}\subset \mathcal{O}_{n}$ of neighborhoods of $y$  and
families $\{U_{n}^{k} |\ k\in\overline{0, 2^{n}-1} \}_{n\in \omega_{0}}$ of regular $2^{n}$-partitions of subspaces $f^{-1} \mathcal{ O}_{n}$, are called {\it a consistent family of binary partitions of a mapping $f: X\to Y$ at a point $y\in Y$} if
\begin{equation*}
	U_{n+1}^{2k}\cup U_{n+1}^{2k+1}=U_{n}^{k}\cap f^{-1}\mathcal{O}_{n+ 1},\ k\in\overline{0, 2^{n}-1}.
\end{equation*}
\end{definition}
	
\begin{proposition} \label{prFContMapOnCorPart} Let for a mapping $f: X\to Y$ and a point $y\in Y$ a consistent family $\{\mathcal{O}_{n}, \{U_{n }^{k}|\ k\in\overline{0, 2^{n}-1}\}|\ n\in\mathbb\omega_0\}$ of binary partitions of a mapping $f: X\to Y$ at $ y\in Y$ is given.
	
Then, for the family of functions
	$\varphi_{n}: X \to [0,1]$
	\begin{equation*}
	\varphi_{n} (x) = \left\{ \begin{array}{ll}
	0, \ & x \in X\setminus f^{-1} \mathcal{O}_{n}, \\
	\frac{k}{2^{n}-1}, \ & x \in U^{k}_{n}, \ k\in\overline{0, 2^{n}-1}, \\
	\end{array}\right.
	\end{equation*}
the pairs $(\mathcal{O}_{n}, \varphi_{n}), \ n \in \omega_{0}$, satisfy the conditions of Lemma {\rm\ref{lemFContFunc}} and define a $f$--continuous at $y\in Y$ function
	$\varphi: X \to [0,1]$
	\begin{equation*}
	\varphi (x) = \left\{ \begin{array}{ll}
	\varphi_{n}(x), \ & x \in f^{-1}\mathcal{O}_{n} \setminus f^{-1}\mathcal{O}_{n+1},\ n\in\mathbb\omega_0,\\
	\lim\limits_{n \to \infty}\varphi_{n}(x), \ & x\in\bigcap\limits_{n=0}^{\infty} f^{-1}\mathcal{O} _{n}.\\
	\end{array}\right.
	\end{equation*}
Moreover, ${\rm osc}_{\varphi}(f^{-1}\mathcal{O}_{n}) \leqslant \frac{1}{2^{n}-1}$.
\end{proposition}

\begin{proof}
Let's check the fulfillment of conditions (b) and (c) of Lemma \ref{lemFContFunc}
		
From the condition ($\ast$) of  Lemma \ref{remCorrPart} for a regular $2^{n}$-partition $ \{ U^{k}_{n}|\ k\in\overline{0, 2^{n }-1}\}$ of the subspace $f^{-1}\mathcal{O}_{n}$ we have
		
\begin{equation*} \label{eq4UrysLema0}
	\text{osc}_{\varphi_{n}}(f^{-1}\mathcal{O}_{n}) \leqslant \frac{1}{2^{n}-1}, \text{ for } n \in \omega_{0}.
\end{equation*}
Thus, $\lim\limits_{n \to \infty} \mathrm{osc}_{\varphi_{n}}(f^{-1}\mathcal{O}_{n}) = 0$, and the condition (b) of Lemma \ref{lemFContFunc} is satisfied.
		
For any point $x\in f^{-1}\mathcal{O}_{n+1}$ we have $|\varphi_{n+1}(x)-\varphi_{n}(x)|=| \frac{k'}{2^{n+1}-1}-\frac{k}{2^{n}-1}|$. Since $k' = 2k + 1$ or $k' = 2k$ and $k\in\overline{0, 2^{n}-1}$, then
\begin{equation} \label{eq4UrysLema}
\begin{split}
|\varphi_{n+1}(x)-\varphi_{n}(x)| & \leqslant \bigg| \frac{2k+1}{2^{n+1}-1}-\frac{k}{2^{n}-1}\bigg| =\bigg| \frac{2^{n}-1-k}{(2^{n+1}-1)(2^{n}-1)}\bigg| \leqslant \frac{1}{2^{n+1}-1}, \\
|\varphi_{n+1}(x)-\varphi_{n}(x)| & \leqslant \bigg| \frac{2k}{2^{n+1}-1}-\frac{k}{2^{n}-1}\bigg| =\bigg| \frac{k}{(2^{n+1}-1)(2^{n}-1)}\bigg| \leqslant \frac{1}{2^{n+1}-1}.
\end{split}
\end{equation}
		 
So from (\ref{eq4UrysLema})
\begin{equation*}
	\sum\limits_{n=0}^{\infty}||\varphi_{n+1}|_{f^{-1}\mathcal{O}_{n+1}} - \varphi_{n} |_{f^{-1}\mathcal{O}_{n+1}}|| \leqslant \sum_{n=0}^{\infty} \frac{1}{2^{n+1}-1},
\end{equation*}
the majorizing series converges and, thus, the condition (c) of Lemma \ref{lemFContFunc} is satisfied.
		
It follows from the definition of $\varphi$ that $\text{osc}_{\varphi}(f^{-1}\mathcal{O}_{n}) \leqslant \text{osc}_{\varphi_{ n}}(f^{-1}\mathcal{O}_{n}) \leqslant \frac{1}{2^{n}-1}$, for $ n \in \omega_{0}$.		
\end{proof}

\begin{lemma} \label{lemSumFcontMapFcont}
	Let a mapping $f: X \to Y$, a point $y\in Y$ and a family of {\rm (}bounded{\rm)} $f$-continuous at $y \in Y$ functions $\varphi_{ n}: X\to\mathbb{R}$, $n \in \mathbb{N}$, be such that the series $\sum\limits_{n=1}^{\infty}||\varphi_{n} ||$ converges.
	Then the function $\varphi (x) = \sum\limits_{n=1}^{\infty}\varphi_{n}(x)$ is $f$-continuous at $y$.
	\end{lemma}
	
\begin{proof}
Since the series $\sum\limits_{n=1}^{\infty}||\varphi_{n}||$ converges, the function $\varphi (x)$ is well defined.
	
Consider an arbitrary $\varepsilon > 0$. Let $N \in \mathbb{N}$ be such that $\sum\limits_{n=N+1}^{\infty}||\varphi_{n}||<\frac{\varepsilon}{4} $. The $f$-continuity of $\varphi_{n}$ yields that there exists a neighborhood $\mathcal{O}_{n}$ of $y$ such that $\mathrm{osc}_{\varphi_{n}} (f^{-1}\mathcal{O}_{n}) < \frac{\varepsilon}{2N}$, $n\leqslant N$.
	
Let $\mathcal{O} = \bigcap\limits_{n=1}^{N}\mathcal{O}_{n}$. Then from Remark~\ref{remOsc}  it follows that
\begin{equation*}
	\begin{aligned}
	\mathrm{osc}_{\varphi}(f^{-1}\mathcal{O}) \leqslant \sum\limits_{n=1}^{N}\mathrm{osc}_{\varphi_{n} }(f^{-1}\mathcal{O}) & + \mathrm{osc}_{\sum\limits_{n=N+1}^{\infty}\varphi_{n}}(f^{- 1}\mathcal{O}) < \\
	< \frac{\varepsilon \cdot N}{2N} & + 2 \sum\limits_{n=N+1}^{\infty}||\varphi_{n}||<\frac{\varepsilon}{2 } + \frac{\varepsilon}{2} = \varepsilon.
	\end{aligned}
\end{equation*}
\end{proof}


\section{Characterizations of normal mappings}

Subsets $A$, $B$ of a space $X$ are called {\it separated by neighborhoods in a subspace} $X^{\prime}\subset X$ \cite[p. 73]{PasOtobrFunct} if the sets $A \cap X^{\prime}$ and $B \cap X^{\prime}$ have disjoint neighborhoods in $X^{\prime}$.
For a mapping $f: X \to Y$,  sets $A, B \subset X$ are called $f-${\it separated by neighborhoods} \cite[p. 73]{PasOtobrFunct}, if any point $y \in Y$ has a neighborhood $\mathcal{O}y$, in the preimage $f^{-1}\mathcal{O}y $ of which the sets $A \cap f^ {-1}\mathcal{O}y$ and $B \cap f^{-1}\mathcal{O}y$ are separated by neighborhoods.

\begin{definition}\cite[p. 73]{PasOtobrFunct}. A mapping $f:X \to Y$ is said to be {\it prenormal} if any two disjoint closed subsets $A$ and $B$ of $X$ are $f$-separated by neighborhoods.

A mapping $f:X \to Y$ is called {\it normal} \cite[p. 52]{MusPas} if for any $\mathcal{O} \in \tau_{Y} $ the restriction $f_{\mathcal{O}}:f^{-1}\mathcal{O} \rightarrow \mathcal{ O}$ of $f$ to $\mathcal{O}$ is prenormal.
\end{definition}

The following statement is a convenient generalization of the ``small Urysohn Lemma'' for a normal mapping~\cite[Proposition 5]{liseev2}.

\begin{lemma}
	Let $f: X \to Y$ be a normal mapping. Then for any $\mathcal{O} \in \tau_{Y}$ and any pair of disjoint closed in $f^{-1}\mathcal{O}$ subsets $F$ and $T$, for any point $y \in \mathcal{O}$ there is a consistent family of binary partitions $\{\mathcal{O}_{n}, \{U_{n}^{k}|\ k\in\overline{0, 2^{n }-1}\}|\ n\in \mathbb\omega_0\}$ of a mapping $f: X \to Y$ at $y$, such that for any $n \in \mathbb{N}$
	\begin{itemize}
	\item[{\rm (a)}] $F\cap f^{-1}\mathcal{O}_n\subset U_{n}^{0}$,\ \ $T\cap f^{-1 }\mathcal{O}_n\subset U_{n}^{2^{n}-1}$\rm;
	\item[{\rm (b)}] $F\cap \mathrm{cl}_{ f^{-1}\mathcal{O}_n}(\bigcup\limits_{k=1}^{2^{ n}-1} U_{n}^{k})=\varnothing,\ \ \mathrm{cl}_{ f^{-1}\mathcal{O}_n}(\bigcup\limits_{k=0} ^{2^{n}-2} U_{n}^{k}) \cap T = \varnothing$.
	\end{itemize}
	\label{lemFSetPartitionofNormalMap}
\end{lemma}

	\begin{proof}
		Without loss of generality,  assume that $\mathcal{O} = Y$, disjoint sets $F,\ T$ are closed in $X$ and $y\in Y$. By induction we construct a consistent family of binary partitions of the mapping $f: X \to Y$ at $y$.
		
		Induction base $n=0$, $\mathcal{O}_{0} = Y$, $U_{0}^{0} = X$.
		
		Induction step. Suppose that for $i \leqslant n$ the neighborhoods $\mathcal{O}_{i}$ of the point $y$, $\mathcal{O}_{i+1} \subset \mathcal{O}_{i }$, and the families of regular $2^{i}$-partitions $\{ \{U^{j}_{i} | j\in\overline{0, 2^{i}-1}\} |\ i\in\overline{0,n}\}$  of the subspaces $f^{-1} \mathcal{O}_{i}$, $i \leqslant n$, satisfying conditions (a), (b) of lemma and condition of Definition~\ref{defCoherenceSetBinPart}, have been constructed.
		
		From the normality of the mapping $f$ and \cite[Proposition 5]{liseev2} the following hold. Firstly, for a closed in 
		$f^{-1}\mathcal{O}_{n}$ subset 
		$\mathrm{cl}_{f^{-1}\mathcal{O}_{n}} \Big( \bigcup\limits_{j=1}^{2^{n}-1} U_{n}^{j} \Big)$ 
		and its neighborhood $f^{-1}\mathcal{O}_{ n} \setminus F$ there is such a neighborhood $\mathcal{O}_{n+1}^{0} \subset \mathcal{O}_{n}$ of $y$ and an open in $f^{- 1}\mathcal{O}_{n+1}^{0}$ subset $\widetilde{V}^{0}$ such that
	\begin{equation}\label{eq1Uryslemma}
		\mathrm{cl}_{f^{-1}\mathcal{O}_{n+1}^{0}} \big(\bigcup\limits_{j=1}^{2^{n}-1 } (U_{n}^{j} \cap f^{-1}\mathcal{O}^0_{n+1})\big) \subset \widetilde{V}^{0} \subset \mathrm{ cl}_{f^{-1}\mathcal{O}_{n+1}^{0}} \widetilde{V}^{0} \subset f^{-1} \mathcal{O}_{ n+1}^{0} \setminus F.
	\end{equation}
		
		Secondly, for $p\in\overline{2, 2^{n}-1}$, a closed in $f^{-1}\mathcal{O}_{n}$ subset
		$\mathrm{cl}_{f^{-1}\mathcal{O}_{n}} \bigcup\limits_{j=p}^{2^{n}-1} U_{n}^{j }$ and its neighborhood $f^{-1}\mathcal{O}_{n} \setminus \bigcup\limits_{j=0}^{p-2} U_{n}^{j}$ there are a neighborhood $\mathcal{O}_{n+1}^{p-1} \subset \mathcal{O}_{n}$ of the point $y$, and an open in $f^{-1}\mathcal{O }_{n+1}^{p-1}$ set $\widetilde{V}^{p-1}$ such that
		\begin{equation}\label{eq2Uryslemma}
			 \begin{aligned}
		\mathrm{cl}_{f^{-1}\mathcal{O}_{n+1}^{p-1}} \big(\bigcup\limits_{j=p}^{2^{n} -1} (U_{n}^{j} \cap f^{-1}\mathcal{O}^{p-1}_{n+1})\big) \subset \widetilde{V}^{ p-1} & \subset \mathrm{cl}_{f^{-1}\mathcal{O}_{n+1}^{p-1}} \widetilde{V}^{p-1} \subset \\
			 & \subset f^{-1}\mathcal{O}_{n} \setminus \bigcup\limits_{j=0}^{p-2} (U_{n}^{j} \cap f^{- 1}\mathcal{O}_{n+1}^{p-1}).
			 \end{aligned}
		\end{equation}

Thirdly, for a closed in $f^{-1}\mathcal{O}_{n}$ subset $T \cap f^{-1}\mathcal{O}_{n}$  and its neighborhood $ U_ {n}^{2^{n}-1} $ there are a neighborhood $\mathcal{O}_{n+1}^{2^{n}-1} \subset \mathcal{O}_{n} $ of point $y$ and an open subset $\widetilde{V}^{2^{ n}-1}$ such that
	\begin{equation}
		\label{eq3Uryslemma}
		T \cap f^{-1}\mathcal{O}_{n+1}^{2^{n}-1}\ \ \subset \widetilde{V}^{2^{n}-1} \subset \ \mathrm{cl}_{f^{-1} \mathcal{O}_{n+1}^{2^{n}-1}} \widetilde{V}^{2^{n} -1} \subset U_{n}^{2^{n}-1} \cap f^{-1}\mathcal{O}_{n+1}^{2^{n}-1}.
		\end{equation}
Consider the sets $\mathcal{O}_{n+1} = \bigcap\limits_{p=0}^{2^{n}-1} \mathcal{O}_{n}^{p}$ , $V^{p} = \widetilde{V}^{p} \cap f^{-1} \mathcal{O}_{n+1}$, $p\in\overline{0, 2^{n }-1}$. For the family of sets $\{V^{p} \}_{p=0}^{2^{n}-1}$ one has:
		
\begin{itemize}
	\item[(i)] directly from the construction it follows that $V^{p} \supset V^{p+1}$, $p\in\overline{0,2^{n}-2}$;
		
	\item[(ii)] from inclusions (\ref{eq1Uryslemma}) and the fulfillment of condition (a) for a regular $2^{n}$-partition $\{ U^{k}_{n} \}_{k=0}^{2^{n}-1}$ of the subspace $f^{-1}\mathcal{O}_{n}$ it follows that
	$$F \cap \mathrm{cl}_{f^{-1}\mathcal{O}_{n+1}} V^{0}=\varnothing,\ \ \mathrm{cl}_{ f^ {-1}\mathcal{O}_{n+1}}(\bigcup\limits_{k=1}^{2^{n}-1} (U_{n}^{k}\cap f^{ -1}\mathcal{O}_{n+1})\big)\subset V^{0};$$
		
	\item[(iii)] from inclusions (\ref{eq2Uryslemma})  it follows for $p\in\overline{0, 2^{n}-3}$ that 
		$$(\bigcup\limits_{k=0}^{p} (U_{n}^{k}\cap f^{-1}\mathcal{O}_{n+1})\big)\cap \mathrm{cl}_{ f^{-1}\mathcal{O}_{n+1}} V^{p+1}=\varnothing,$$
				 $$ \mathrm{cl}_{ f^{-1}\mathcal{O}_{n+1}}(\bigcup\limits_{k=p+2}^{2^{n}-1} ( U_{n}^{k}\cap f^{-1}\mathcal{O}_{n+1})\big)\subset V^{p+1};$$
		
	\item[(iv)] from inclusions (\ref{eq3Uryslemma}) and the fulfillment of condition (b) for a regular $2^{n}$-partition\-$\{ U^{k}_{n} \} _{k=0}^{2^{n}-1}$ of the subset $f^{-1}\mathcal{O}_{n}$ it follows that
		\begin{equation*}
		(\bigcup\limits_{k=0}^{2^{n}-2} (U_{n}^{k}\cap f^{-1}\mathcal{O}_{n+1})\big)\cap\mathrm{cl}_{ f^{-1}\mathcal{O}_{n+1}} V^{2^{n}-1} = \varnothing \text{ and } T \cap f^{-1}\mathcal{O}_{n+1}\subset V^{2^{n}-1}.
		\end{equation*}
		\end{itemize}

Let $U_{n+1}^{2k}= (U_{n}^{k} \setminus V^{k})\cap f^{-1}\mathcal{O}_{n+1}$ , $ U_{n+1}^{2k+1}=(U_{n}^{k}\cap V^{k})\cap f^{-1}\mathcal{O}_{n+1 }$, $k\in\overline{0, 2^{n}-1}$. Then from (ii) -- (iv) it follows that $\{U^{k}_{n+1} |\ k\in\overline{0, 2^{n+1}-1}\}$ is a regular $2^{n+1}$-partition of the subspace $f^{-1} \mathcal{O}_{n+1}$ for which the condition of Definition~\ref{defCoherenceSetBinPart} is satisfied.

The constructed sequence of neighborhoods $\mathcal{O}_{n}$ and families of regular $2^{n}$-partitions of the subspace $f^{-1} \mathcal{O}_{n}$, $n\in\omega_0 $ are the consistent family $\{\mathcal{O}_{n}, \{U_{n}^{k}| \ k\in\overline{0, 2^{n}-1}\}|\ n\in\mathbb\omega_0\}$ of binary partitions of the mapping $f: X \to Y$ at $y$, for which the fulfillment properties (a) and (b) follows from (ii) and (iv).
		\end{proof}

		\begin{theorem} \label{thMainFuncChar}
			For a mapping $f: X \to Y$ the following conditions are equivalent.
			\begin{enumerate}
			\item[{\rm (A)}] The mapping $f$ is normal;
			\item[{\rm (B)}] For any $\mathcal{O} \in \tau_{Y}$ and any pair of disjoint closed in $f^{-1}\mathcal{O}$ subsets $F$ and $T$, for any point $y \in \mathcal{O}$ there is a consistent family of binary partitions $\{\mathcal{O}_{n}, \{U_{n}^{k}|\ k\in\overline{0, 2^{n}-1}\}|\ n\in\mathbb\omega_0\}$ of the mapping $f: X \to Y$ at $y$, such that for any $n \in \mathbb{N}$
			\begin{itemize}
			\item[{\rm (a)}] $F\cap f^{-1}\mathcal{O}_n\subset U_{n}^{0}$,\ \ $T\cap f^{-1 }\mathcal{O}_n\subset U_{n}^{2^{n}-1}$\rm;
			\item[{\rm (b)}] $F\cap \mathrm{cl}_{ f^{-1}\mathcal{O}_n}(\bigcup\limits_{k=1}^{2^{ n}-1} U_{n}^{k})=\varnothing,\ \ \mathrm{cl}_{ f^{-1}\mathcal{O}_n}(\bigcup\limits_{k=0} ^{2^{n}-2} U_{n}^{k}) \cap T = \varnothing$.
			\end{itemize}
			\item[{\rm (C)}]
			For any $\mathcal{O}\in\tau_{Y}$ and any pair of disjoint closed in $f^{-1}\mathcal{O}$ subsets $F$ and $T$, for any point $y \in \mathcal{O}$ there exist a $f$-continuous at $y$ function $\varphi: X \to [0,1]$ and a neighborhood $\mathcal{O}y$ of $y$ such that $\mathrm{osc}_{ \varphi}(f^{-1}\mathcal{O}y)<\frac{1}{2}$ and
			
			\begin{center}
			$F\cap f^{-1}\mathcal{O}y\subset\varphi^{-1}(0)\cap f^{-1}\mathcal{O}y$,
			$T\cap f^{-1}\mathcal{O}y\subset\varphi^{-1}(1)\cap f^{-1}\mathcal{O}y$;
			
			\medskip
			
			$F\cap f^{-1}\mathcal{O}y\subset f^{-1}\mathcal{O}y\setminus \mathrm{cl}_{f^{-1}\mathcal{O} y} (\varphi^{-1}[\frac12, 1]\cap f^{-1}\mathcal{O}y)$,
			$T\cap f^{-1}\mathcal{O}y\subset\mathrm{int}_{f^{-1}\mathcal{O}y}(\varphi^{-1}[\frac12, 1]\cap f^{-1}\mathcal{O}y)$.
			\end{center}
			
			\item[{\rm (D)}] Let $\mathcal{O} \in \tau_{Y}$, $F$ is a non-empty closed subset of $f^{-1}\mathcal{O}$ and $y \in \mathcal{O}$ is arbitrary. Then, for any $\widetilde{f}$-continuous at $y$ function $\widetilde{\varphi}: F \to \mathbb{R}$ of the mapping $\widetilde{f}: F \to \mathcal {O}$, {\rm(}$\widetilde{f}(x) = f(x)$ for $x \in F${\rm)} there is a $f$-continuous at $y$ function $\varphi: X\to\mathbb{R}$  for the mapping $f: X \to Y$ such that
			\begin{itemize}
			\item[{\rm (a)}]$\widetilde\varphi|_{F\cap f^{-1}G}=\varphi|_{F\cap f^{-1}G}$, where $ G $ is a $ G_{\delta}$--subset of $\mathcal{O}$ and $y \in G$ \\
			{\rm(}in particular, $\widetilde\varphi|_{F\cap f^{-1}y}=\varphi|_{F\cap f^{-1}y}${\rm)} ,
			
			\item[{\rm (b)}]$||\varphi||\leqslant||\widetilde{\varphi}||$,
			
			\item[{\rm (c)}] for any $\varepsilon > 0$ there is a neighborhood $\mathcal{O}(\varepsilon) \subset \mathcal{O}$ of $y$ such that \\ $| |\widetilde{\varphi}|_{F\cap f^{-1}\mathcal{O}(\varepsilon)} - \varphi|_{F\cap f^{-1}\mathcal{O}( \varepsilon)}|| <\varepsilon$.
			\end{itemize}
			\end{enumerate}
			\end{theorem}

\begin{proof}
\textbf{(A) $\Rightarrow $ (B)} by Lemma~\ref{lemFSetPartitionofNormalMap}.

\textbf{(B) $\Rightarrow $ (C).}
By (B) for a pair of disjoint closed in $f^{-1}\mathcal{O}$ subsets $F$ and $T$ there exists a consistent family of binary partitions $\{ \mathcal{O}_{n}, \{U^{k}_{n} |\ k\in\overline{0,2^{n}-1} \}|\ n \in \omega_{0} \}$ of the mapping $f:X\to Y$ at $y$ for which conditions (a) and (b) hold.

In Proposition~\ref{prFContMapOnCorPart} the $f$-continuous at $y\in Y$ function 
\begin{equation*}
\varphi (x) = \left\{ \begin{array}{ll}
\varphi_{n}(x), & x \in f^{-1}\mathcal{O}_{n} \setminus f^{-1}\mathcal{O}_{n+1},\ n \in\mathbb\omega_0,\\
\lim\limits_{n \to \infty}\varphi_{n}(x), & x\in\bigcap\limits_{n=0}^{\infty} f^{-1}\mathcal{O}_ {n}. \\
\end{array}\right.
\end{equation*}
according to a consistent family of binary partitions $\{ \mathcal{O}_{n}, \{ U^{k}_{n} | \ k\in\overline{0,2^{n}-1} \}|\ n \in \omega_{0} \}$ of the mapping $f:X\to Y$ at $y$ was constructed, where $\varphi_{n}: X \to \mathbb{R}$ are the following
\begin{equation*}
\varphi_{n} (x) = \left\{ \begin{array}{ll}
0, & x \in X\setminus f^{-1} \mathcal{O}_{n}, \\
\frac{k}{2^{n}-1}, & x \in U^{k}_{n}, \ k\in\overline{0, 2^{n}-1},\ n \in \omega_{0}.\\
\end{array}\right.
\end{equation*}

In this case, firstly, $\varphi (x) = 0$ for $x \in F$ (since $F \subset U_{n}^{0}$, $n \in \mathbb{N} $), $\varphi (x) = 1$ for $x \in T$ (since $T \subset U_{n}^{2^{n}-1}$, $ n \in \mathbb{N}$).

Secondly, due to the $f$-continuity of $\varphi$ there is a neighborhood $\mathcal{O}y$ of $y$ such that $\mathrm{osc}_{\varphi}(f^{-1} \mathcal{O}y)<\frac{1}{2}$.
Then by Lemma \ref{lemOscDisj}
\begin{equation*}
F \cap f^{-1}\mathcal{O}y\subset \varphi^{-1}(0)\cap f^{-1}\mathcal{O}y\subset f^{-1} \mathcal{O}y \setminus \mathrm{cl}_{ f^{-1}\mathcal{O}y}(\varphi^{-1}[\tfrac{1}{2}, 1]\cap f ^{-1}\mathcal{O}y).
\end{equation*}

Thirdly, for any $n\in \mathbb{N}$
\begin{equation*}
\varphi_n^{-1}[\tfrac{1}{2}, 1]=\bigcup\limits_{k=2^{n-1}}^{2^{n}-1} U_{n}^ {k}=U^{1}_{1}\cap f^{-1}\mathcal{O}_n.
\end{equation*}
Hence,
\begin{equation*}
T\cap f^{-1}\mathcal{O}y\subset U^{1}_{1} \cap f^{-1}\mathcal{O}y\subset\mathrm{int}_{ f ^{-1}\mathcal{O}y} (\varphi^{-1}[\tfrac{1}{2}, 1]\cap f^{-1}\mathcal{O}y).
\end{equation*}

\textbf{(C) $\Rightarrow$ (D).} If $Y\ne\mathcal{O}$ and the $f_{\mathcal{O}}$-continuous at  $y$ function $\varphi : f^{-1}\mathcal{O}\to \mathbb{R}$ (for a restriction $f_{\mathcal{O}}: f^{-1}\mathcal{O}\to \mathcal{O}$ of the mapping $f: X \to Y$) which satisfies conditions  similar to the conditions of (D) is constructed, then defining, additionally, the function $\varphi$ by the value 0 at the points $X\setminus f^{-1}\mathcal{O}$ we obtain the required function. Therefore, one can assume that $Y=\mathcal{O}$, the subset $F$ is closed in $X$ and $y \in Y$.

If $y\not\in\mathrm{cl} \big( f(X) \big)$, then the constant function $\varphi$ taking the value 0 on $X$ is the required one.
 
Otherwise, put $\varphi_{0} = \widetilde{\varphi}$, $\mu_{0} = ||\varphi_{0}||$. If $||\varphi_{0}|| = 0$, then put $\varphi \equiv 0$.

Otherwise, from the $\widetilde{f}$-continuity of the function $\widetilde{\varphi}$ at $y$, let $\mathcal{O}'_{0}$ be a neighborhood of $y$ such that $ \mathrm{osc}_{\widetilde{\varphi}}(F\cap f^{-1}\mathcal{O}'_0)<\frac{\mu_0}{3}$,
\begin{equation*}
\begin{split}
P_{0} = \text{cl}_{ f^{-1}\mathcal{O}'_{0}}\{x \in F \cap f^{-1}\mathcal{O}'_ {0}\ |\ \varphi_{0}(x) \leqslant - \tfrac{\mu_{0}}{3}\}, \\
Q_{0} = \text{cl}_{f^{-1}\mathcal{O}'_{0}}\{ x \in F \cap f^{-1}\mathcal{O}'_ {0}\ |\ \varphi_{0}(x) \geqslant \tfrac{\mu_{0}}{3}\}.
\end{split}
\end{equation*}
Then, by Lemma~\ref{lemOscDisj}, the closed in $f^{-1}\mathcal{O}'_{0}$ subsets $P_{0}$ and $Q_{0}$ are disjoint. By (C) (replacing the segment $[0, 1]$ with $[-\frac{\mu_{0}}{3}, \frac{\mu_0}{3}]$), there exist a $f$-continuous at  $y$ function $\psi_{0}: X \to [-\frac{\mu_{0}}{3}, \frac{\mu_{0}}{3}]$
and a neighborhood $\mathcal{O}_{0} \subset \mathcal{O}'_{0}$ of the point $y$ such that
\begin{equation*}
	\begin{split}
		P_{0} \cap f^{-1}\mathcal{O}_{0} \subset\psi^{-1}_{0}(-\tfrac{\mu_{0}}{3}),\ Q_{0}\cap f^{-1}\mathcal{O}_{0} \subset \psi^{-1}_0(\tfrac{\mu_{0}}{3}), \\
P_{0} \cap f^{-1}\mathcal{O}_{0} \subset f^{-1}\mathcal{O}_{0} \setminus \mathrm{cl}_{ f^{-1}\mathcal{O}_{0}}(\psi_0^{-1}([0, \tfrac{\mu_{0}}{3}]) \cap f^{-1}\mathcal{O}_{0}), \\
Q_{0} \cap f^{-1}\mathcal{O}_{0} \subset \mathrm{int}_{ f^{-1}\mathcal{O}_{0}}(\psi_{0}^{-1}([0, \tfrac{\mu_{0}}{3}]) \cap f^{-1}\mathcal{O}_{0}), \\
\mathrm{osc}_{\psi_{0}}(F\cap f^{-1}\mathcal{O}_{0}) < \tfrac{\mu_{0}}{3}.
	\end{split}
\end{equation*}
Let $\varphi_{1}=\varphi_{0} - \psi_{0}: F\cap f^{-1}\mathcal{O}_{0} \to \mathbb{R}$. By Remark~\ref{remFcontMap} the function $\varphi_{1}$ is $f|_{F \cap f^{-1}\mathcal{O}_{0}}$--continuous at $y$ and $||\varphi_{1}|| = \mu_{1} \leqslant \tfrac{2\mu_{0}}{3}$.
 
By induction we construct a non-increasing (with respect to embedding) sequence of neighborhoods $\{\mathcal{O}_{n}\}_{n=0}^{\infty}$ of $y$, a sequence of $f$-continuous at $y$ functions $\psi_{n}: X\to \mathbb{R}$ and a $\widetilde{f}$-continuous at $y$ functions $\varphi_{n}: F\cap f^{-1 }\mathcal{O}_{n} \to \mathbb{R}$ such that $\varphi_{n+1} = \varphi_{n}-\psi_{n}$, $||\psi_{n }|| \leqslant \frac{\mu_{n}}{3}$, $||\varphi_{n+1}|| = \mu_{n+1} \leqslant \frac{2\mu_{n}}{3}$.
\begin{equation}
||\varphi_{n}|| \leqslant \bigg(\frac{2}{3} \bigg)^{n}\mu_{0},\ ||\psi_{n}||\leqslant \bigg(\frac{2}{3} \bigg)^{n} \frac{\mu_{0}}{3}.\
\label{eqno(2.4)}
\end{equation}
 
Let
\begin{equation} \label{eqno(2.4.1)}
\varphi(x) = \sum\limits_{k=0}^{\infty}\psi_{n}(x).
\end{equation}
 
From the conditions (\ref{eqno(2.4)}) it follows that
the series $\sum\limits_{k=0}^{\infty}||\psi_{n}||$ converges. By Lemma
\ref{lemSumFcontMapFcont} the function $\varphi: X\to \mathbb{R}$ is $f$-continuous at $y$.
The inequality $||\varphi|| \leqslant ||\widetilde{\varphi}||$ follows from the second inequality of (\ref{eqno(2.4)}). Indeed, from (\ref{eqno(2.4.1)}) it follows that
\begin{equation*}
||\varphi|| \leq \sum\limits_{k=0}^{\infty}||\psi_{n}|| \leqslant \frac{\mu_{0}}{3} \sum\limits_{k=0}^{\infty} \bigg(\frac{2}{3} \bigg)^{k} = ||\widetilde{\varphi}||.
\end{equation*}
From the equality (\ref{eqno(2.4.1)}), equation $\varphi_{n+1}=\varphi_{n}-\psi_{n} = \widetilde{\varphi} - \sum\limits_ {k=0}^{n}\psi_{n}$
and the first inequality of (\ref{eqno(2.4)}) one has 
$\widetilde{\varphi}|_{F \cap \bigcap\limits_{n=0}^{\infty} f^{-1}\mathcal{O}_{n}} = \varphi|_{F \cap \bigcap\limits_{n=0}^{\infty}f^{-1}\mathcal{O}_{n}}$ (in particular, $\widetilde{\varphi}|_{F \cap f^{-1}y} = \varphi|_{F \cap f^{-1}y}$) and the last condition of item (D) holds.

\textbf{(D) $\Rightarrow$ (A).} Let $\mathcal{O}\in\tau_Y$, $F$ and $T$ are closed disjoint subsets of $f^{-1}\mathcal{O}$, $y\in \mathcal{O}$.

The function
\begin{equation} \label{eqBTU1}
\widetilde{\varphi} (x)=\left\{\begin{array}{cl}
0, & x \in F, \\
1, & x \in T \\
\end{array}\right.
\end{equation}
is $\widetilde{f}$-continuous at $y$ for the restriction $\widetilde{f}: F\cup T\to {\mathcal{O}}$ of the mapping $f: X \to Y$. The set $F\cup T$ is closed in $f^{-1}\mathcal{O}$.

By condition (D) there is a
$f$-continuous at $y$ function $\varphi: X\to [0, 1]$ for the mapping $f: X \to Y$ such that $\widetilde\varphi|_{(F \cup T) \cap f^{-1}G}=\varphi|_{(F\cup T)\cap f^{-1}G}$, where $ G $ is a $G_{\delta}$-subset of $\mathcal{O}$ and $y \in G$ (in particular $\widetilde\varphi|_{(F \cup T)\cap f^{-1}y}=\varphi|_{ (F \cup T)\cap f^{-1}y}$), $||\varphi||\leqslant||\widetilde{\varphi}||$, and for $\varepsilon=\frac{1} {4}$ there is a neighborhood $\mathcal{O}'\subset \mathcal{O}$ of $y$ such that $||\widetilde{\varphi}|_{(F \cup T)\cap f^ {-1}\mathcal{O}'} - \varphi|_{(F \cup T)\cap f^{-1}\mathcal{O}'}|| <\frac{1}{4}$ .

Since the function $\varphi: X\to\mathbb{R}$ is $f$-continuous at $y$ for the mapping $f: X \to Y$, then there exists a neighborhood $\mathcal{O}y\subset \mathcal{O}'$ such that
$$\mathrm{osc}_{\varphi}(f^{-1}\mathcal{O}y) < \frac{1}{4}.$$
Then, by Lemma~\ref{lemOscDisj} we have
$\big(\varphi^{-1}[0, \frac{1}{4}] \cap f^{-1}\mathcal{O}y \big) \cap \big( \mathrm{cl} _{f^{-1}\mathcal{O}y} ( \varphi^{-1}[\frac{3}{4}, 1] \cap f^{-1}\mathcal{O}y) \big) = \varnothing$. Open sets $V = \mathrm{int}_{f^{-1}\mathcal{O}y} \big(\varphi^{-1}[0, \frac{1}{4}] \cap f^ {-1}\mathcal{O}y\big)$ and $ U = \mathrm{int}_{f^{-1}\mathcal{O}y} \big(\varphi^{-1}[\frac{3}{4}, 1] \cap f^{-1}\mathcal{O}y\big)$ are disjoint. From (\ref{eqBTU1}) it follows that $F\cap f^{-1}\mathcal{O}y \subset V$, $T\cap f^{-1}\mathcal{O}y\subset U$. 
Thus, the mapping $f: X \to Y$ is normal.
\end{proof}

\begin{remark}\label{remUryshonlemma}
	{\rm 1.} Implication {\rm (A) $\Rightarrow $ (C)} of Theorem~{\rm\ref{thMainFuncChar}} is the Urysohn's Lemma for mappings, which is formulated in~\cite{Zubov}.
	
	{\rm 2.} Statement {\rm (D)} of Theorem~{\rm\ref{thMainFuncChar}} is a variant of extension of the Brouwer-Tietze-Urysohn Theorem to mappings.
	 
	{\rm 3.} The extension of a $f|_{F}$-continuous at $y$ function, is equivalent to the extension of a $f|_{F}$-continuous at $y$ bounded mapping into a Banach space.
	
	{\rm 4.} In the case of a constant mapping $f: X \to \{y\}$ the statement of Theorem~{\rm\ref{thMainFuncChar}} coincides with the statement of Urysohn's Lemma~\cite[p. 75, Theorem 1.5.10]{Engelking} and the Brouwer--Tietze--Urysohn Theorem\cite[p. 116, Theorem 2.1.8]{Engelking} for spaces.
\end{remark}


\section{Family of functions $f$-equicontinuous at a point and $\sigma$-normality of a mapping}

\begin{definition}\label{defEquicontMapSeq}
For a mapping $f: X \to Y$, a family of functions $\{\varphi_{n}: X \to [0, 1]\}_{n \in \mathbb{N}}$ is called a {\it $f$-equicontinuous family of functions at a point $y\in Y$} if for any $\varepsilon > 0$ there is a neighborhood $\mathcal{O}y$ of $y$ such that for any $n \in \mathbb{N}$
\begin{equation*}
 \mathrm{osc}_{\varphi_{n}}(f^{-1}\mathcal{O}y) < \varepsilon.
\end{equation*}
\end{definition}

\begin{remark}
	Each function $\varphi_{n}$ of a $f$--equicontinuous family $\{\varphi_{n}\}_{n \in \mathbb{N}}$ of functions at $y$ is $f$--continuous at $y$, the converse is not true.
\end{remark}
	
\begin{definition} \cite[Definition 7]{liseev3}.
	\label{defSigmNorm}
	A mapping $f: X \rightarrow Y$ is said to be a {\it $\sigma$--prenormal} if for any $F_{\sigma}$--set $T=\bigcup\limits_{l=1}^ {\infty} T_{l}$, where the subset $T_{l}$ is closed in $X$, $l \in \mathbb{N}$, and a closed in $X$ subset $F$ such that $ T \cap F = \varnothing$, for any point $y \in Y$ there are its neighborhood $\mathcal{O}y$ and a family $\{\mathcal{O}_{l}\}_{l=1}^{\infty}$ of open in $f^{-1}\mathcal{O}y$ sets such that $T_{l} \cap f^{-1}\mathcal{O}y \subset \mathcal{O}_{l}$, $l \in \mathbb{N}$, and	$\Big(\bigcup\limits_{l=1}^{\infty} \mathrm{cl}_{f^{-1}\mathcal{O}y}(\mathcal{O}_{l}) \Big) \cap F = \varnothing $.
	
	A mapping $f$ is called a {\it $\sigma$--normal} if the restriction $f_{\mathcal{O}}: f^{-1} \mathcal{O} \to \mathcal{O} $ of $f$ to $ \mathcal{O}$ is $\sigma$--prenormal for any $\mathcal{O} \in \tau_{Y}$.
\end{definition}
	
It should be noted that in the case of a constant mapping its prenormality, normality, $\sigma$--prenormality and $\sigma$--normality are equivalent. However, unlike normality, $\sigma$--normality of the mapping is inherited by $F_{\sigma}$--submappings \cite[Theorem 6]{liseev3}.

\begin{lemma}
	A mapping $f: X \to Y$ $\sigma$ is normal iff for any $\mathcal{O} \in \tau_{Y}$,
	any point $y \in \mathcal{O}$ and any $F_{\sigma}$-subset $T = \bigcup\limits_{l=1}^{\infty} T_{l}$ in $f^{ -1}\mathcal{O}$, and its neighborhood $U \subset f^{-1}\mathcal{O}$, there is a neighborhood $\mathcal{O}y \subset \mathcal{O}$ of the point $y$ such that in $f^{-1}\mathcal{O}y $ there are neighborhoods $V_{l}$ of the sets $T_{l} \cap f^{-1}\mathcal{O}y$, $ l \in \mathbb{N}$, and 
	\begin{equation*}
	T_{l} \cap f^{-1}\mathcal{O}y \subset V_{l} \subset {\rm cl}_{f^{-1}\mathcal{O}y}V_{l} \subset U \cap f^{-1}\mathcal{O}y, \ l \in \mathbb{N}.
	\end{equation*}
	\label{lowUryslemmaSigmaNorm}
\end{lemma}
	
\begin{proof}
	{\it Necessity. } Let us fix an arbitrary subset $\mathcal{O} \in \tau_{Y}$ and a point $y \in \mathcal{O}$. For closed in $f^{-1} \mathcal{O}$ sets $T_{l}$, $l \in \mathbb{N}$, and their neighborhood $U$ consider the set $f^{-1}\mathcal{O}\setminus U$. It is closed in $f^{-1}\mathcal{O}$ and $T \cap (f^{-1}\mathcal{O}\setminus U) = \varnothing$.
	Since the mapping $f$ is $\sigma$-normal, then for $y$ there are a neighborhood $\mathcal{O}y \subset \mathcal{O}$ and neighborhoods $V_{l} \subset f^{- 1} \mathcal{O}y $ of subsets $T_{l} \cap f^{-1} \mathcal{O}y$ such that
	\begin{equation*}
  \bigcup\limits_{l=1}^{\infty} \mathrm{cl}_{f^{-1} \mathcal{O}y}(V_{l}) \cap (f^{-1} \mathcal{O} y \setminus U) = \varnothing, \ l \in \mathbb{N}.
	\end{equation*}
	Thus, for each $l \in \mathbb{N}$ we have $T_{l} \cap f^{-1}\mathcal{O}y \subset V_{l} \subset {\rm cl}_{ f^{-1}\mathcal{O}y}V_{l} \subset U \cap f^{-1}\mathcal{O}y$.
  
	{\it Sufficiency.} Let $\mathcal{O} \in \tau_{Y}$, $y \in \mathcal{O}$  an arbitrary fixed point  and consider
	two disjoint subsets $F$ and $T$ of the set $f^{-1}\mathcal{O}$ such that
   
  $T=\bigcup\limits_{l=1}^{\infty} T_{l}$ is $F_{\sigma}$--subset, and the subsets $T_{l}$ are closed in $f^{ -1}\mathcal{O}$, $l \in \mathbb{N}$,
  
  $F$ is closed in $f^{-1}\mathcal{O}$.
  
By the assumption of the theorem, for the point $y$, subsets $T_{l}$, $l \in \mathbb{N}$ and their neighborhood $f^{-1}\mathcal{O} \setminus F$ there exist a neighborhood $ \mathcal{O}y$ of $y$ and open sets $V_{l}$, $l\in\mathbb{N}$ such that 
  $$T_{l} \cap f^{-1}\mathcal{O}y \subset V_{l} \subset {\rm cl}_{f^{-1}\mathcal{O}y}V_{ l} \subset f^{-1}\mathcal{O}y \setminus F, \ l \in \mathbb{N}.$$ Therefore, $\bigcup\limits_{l=1}^{\infty} \mathrm {cl}_{f^{-1}\mathcal{O}y}(V_{l}) \cap F = \varnothing$. The $\sigma$-prenormality of the mapping $f_{\mathcal{O}}: f^{-1}\mathcal{O}\rightarrow \mathcal{O}$ is proven and the mapping $f$ is normal.
\end{proof}

In~\cite{PasOtobrFunct} {\it the submapping} $f\big|_{X_{0}}: X_{0}\to Y$ is defined as the restriction of the mapping $f: X\to Y$ to a subset $X_ {0} \subset X$. The submapping $f\big|_{X_{0}}: X_{0}\to Y$ is called an {\it open {\rm(}closed{\rm)} submapping} \cite{PasOtobrFunct} if $X_{0}$ is an open (closed) subset of $X$. By {\it disjoint submappings} $f|_{A}: A \to Y$, $f|_{B}: B \to Y$ of the mapping $f: X \to Y$ we understand that the subsets $A$ and $B$ are disjoint.

\begin{definition} \cite[Definition 6]{liseev3}.
\label{defFsimaMap}
A submapping $f|_{X_{0}}:X_{0} \to Y $ is said to be of $F_{\sigma}$-{\it type} {\rm (}or is a {\it $F_{\sigma}$--submapping}{\rm )} if for any point $y \in Y$ there is its neighborhood $\mathcal{O}y \subset Y$ such that
$(f\big|_{X_{0}})^{-1} \mathcal{O}y$ is a $F_{\sigma}$--subset of $f^{-1} \mathcal{O}y $.
\end{definition}

A convenient generalization of the “small Urysohn Lemma” for $\sigma$-normal mappings is the following. 

\begin{lemma}
	Let $f: X \to Y$ be a $\sigma$--normal mapping. Then for
	any $\mathcal{O}\in\tau_{Y}$ and any disjoint closed submapping $f_{\mathcal{O}}|_{F}: F \to \mathcal{O}$ and $F_{\sigma }$-submapping $f_{\mathcal{O}}|_{T}: T\to Y$ of the mapping $f_{\mathcal{O}}: f^{-1}\mathcal{O} \to \mathcal {O}$ and any point $y \in \mathcal{O}$
	there are consistent families of binary partitions $\Gamma_{l}=\{\mathcal{O}_{n}, \{U_{n}^{k}(l)|\ k\in\overline{0, 2^{n }-1}\}|\ n\in\omega_0\}$, $l \in \mathbb{N}$, of a mapping $f: X \to Y$ at $y$ such that for any $n \in\mathbb{N}$
	\begin{itemize}
		\item[{\rm (a)}] $T \cap f^{-1}\mathcal{O}_{1} = \bigcup\limits_{l=1}^{\infty} T_{l}$ , where $T_{l}$ is closed in $f^{-1}\mathcal{O}_{1}$, $l \in \mathbb{N}$\rm;
		\item[{\rm (b)}] $F\cap f^{-1}\mathcal{O}_n\subset U_{n}^{0}(l)$,\ \ $T_{l} \cap f^{-1}\mathcal{O}_n\subset U_{n}^{2^{n}-1}(l)$, $l \in \mathbb{N}$\rm;
		\item[{\rm (c)}] $F\cap \mathrm{cl}_{ f^{-1}\mathcal{O}_n}(\bigcup\limits_{k=1}^{2^{ n}-1} U_{n}^{k}(l))=\varnothing,\ \ \mathrm{cl}_{ f^{-1}\mathcal{O}_n}(\bigcup\limits_{k =0}^{2^{n}-2} U_{n}^{k}(l)) \cap T_{l} = \varnothing$, $l \in \mathbb{N}$.
	\end{itemize}
	
	\label{lemFSetPartitionOfSigmaNormalMap}
\end{lemma}

\begin{proof}
	Without loss of generality, we assume that $\mathcal{O} = Y$, the sets $F,\ T$ are disjoint in $X$ and $y\in Y$.
	Let $\mathcal{O}_{0} = Y$, $U_{0}^{0}(l) = X$, $l \in \mathbb{N}$. Let's construct by induction the consistant families of binary partitions $\Gamma_l$, $l\in\mathbb N$, for the mapping $f: X \to Y$ at the point $y$.
	
	Induction base $n=1$. Since $f|_{T}$ is a $F_{\sigma}$--submapping, there exists a neighborhood $\mathcal{O}'_{1} \in \tau_{Y}$ of $y$ such that $ T\cap f^{-1}\mathcal{O}'_{1} = \bigcup\limits_{l=1}^{\infty} T_{l}$, where $T_{l}$ is closed in $ f^{-1}\mathcal{O}'_{1}$, $l\in \mathbb{N}$. By Lemma~\ref{lowUryslemmaSigmaNorm} there is a neighborhood $\mathcal{O}y \subset \mathcal{O}$ of $y$ such that in $f^{-1}\mathcal{O}y $ there exist  neighborhoods $ V_{l}$ of the sets $T_{l} \cap f^{-1}\mathcal{O}y$, $l \in \mathbb{N}$ such that
	\begin{equation*}
	T_{l} \cap f^{-1}\mathcal{O}y \subset V_{l} \subset {\rm cl}_{f^{-1}\mathcal{O}y}V_{l} \subset U \cap f^{-1}\mathcal{O}y, \ l \in \mathbb{N}.
	\end{equation*}
	Let $\mathcal{O}_1=\mathcal{O}y$, $U^{0}_{1}=f^{-1}\mathcal{O}_1\setminus V_l$, $U^1_1= f^{-1}\mathcal{O}_1\cap V_l$, $l\in\mathbb N$.
	Thus, the validity of condition (a) is established.
	
	Induction step. Suppose that for $i \leqslant n$ and for any $l \in \mathbb{N}$ the neighborhoods $\mathcal{O}_{i}$ of the point $y$, $\mathcal{O}_{i+1} \subset \mathcal{O}_{i} $, $i < n$, and the families of regular $2^{i}$-partitions $\{ \{U^{j}_{i}(l) | j\in\overline{0, 2^{i}-1}\} |\ i\in\overline{0,n}\}$ of subspaces $f^{-1} \mathcal{O}_{i}$, satisfying conditions (b), (c) of the lemma and condition of the Definition~\ref{defCoherenceSetBinPart} have been constructed.
	
	From the $\sigma$-normality of the mapping $f$, using Lemma~\ref{lowUryslemmaSigmaNorm}, we have the following.
	
	Firstly, for the $F_{\sigma}$--subset
	$\bigcup\limits_{l=1}^{\infty}\mathrm{cl}_{f^{-1}\mathcal{O}_{n}} \Big( \bigcup\limits_{j=1} ^{2^{n}-1} U_{n}^{j}(l) \Big)$ of $f^{-1}\mathcal{O}_{n}$ and its neighborhood $f^{-1}\mathcal{O}_{n} \setminus F$ there are a neighborhood $\mathcal{O}_{n+1}^{0} \subset \mathcal{O}_{n}$ of $y$ and open in $f^{-1}\mathcal{O}_{n+1}^{ 0} $ subsets $\widetilde{V}^{0}(l)$, $l \in \mathbb{N}$ such that 
	\begin{equation}\label{eqno(2.6.1.1)}
	\mathrm{cl}_{f^{-1}\mathcal{O}_{n+1}^{0}} \big(\bigcup\limits_{j=1}^{2^{n}-1 } (U_{n}^{j}(l) \cap f^{-1}\mathcal{O}^0_{n+1})\big) \subset \widetilde{V}_{l}^{ 0} \subset \mathrm{cl}_{f^{-1}\mathcal{O}_{n+1}^{0}} \widetilde{V}_{l}^{0} \subset f^ {-1} \mathcal{O}_{n+1}^{0} \setminus F, \ l \in \mathbb{N}.
	\end{equation}

Secondly, provided that $p\in\overline{2,2^{n}-1}$, for the $F_{\sigma}$-subset $\bigcup\limits_{i=1}^{\infty}\big(\mathrm{cl}_{f^{-1}\mathcal{O}_{n}} \big(\bigcup\limits_{m=p}^{2^{n}-1} U_{n }^{m}(l)\big)\big)$ of $f^{-1}\mathcal{O}_{n}$, and its neighborhood
	$
	\bigcup\limits_{i=1}^{\infty}\big(\bigcup\limits_{m=p-1}^{2^{n}-1} (U_{n}^{m}(l) \cap f^{-1}\mathcal{O}_{n})\big)
	$
there are a neighborhood $\mathcal{O}_{n+1}^{p-1} \subset \mathcal{O}_{n}$ of $y$, and open in $f^{-1}\mathcal{O}_{n+1}^{p-1}$ subsets $\widetilde{V}^{p-1}_{l}\subset\bigcup\limits_{m=p-1}^{2 ^{n}-1} U_{n}^{m}(l)$, $ l \in \mathbb{N}$, such that
\begin{equation}
	\begin{aligned}
		\mathrm{cl}_{f^{-1}\mathcal{O}_{n+1}^{p-1}}\big(\bigcup\limits_{m=p}^{2^{n} -1} (U_{n}^{m}(l)\cap f^{-1}\mathcal{O}^{p-1}_{n+1}\big)\big)\subset \widetilde {V}^{p-1}_{l}\subset \mathrm{cl}_{f^{-1}\mathcal{O}_{n+1}^{p-1}} \widetilde{V }^{p-1}_{l}\subset \\
	\subset \bigcup\limits_{i=l}^{\infty}\big(\bigcup\limits_{m=p-1}^{2^{n}-1} (U_{n}^{m}( l)\cap f^{-1}\mathcal{O}_{n+1}^{p-1})\big), \ l \in \mathbb{N}.
	\label{eqno(2.6.1.2)}
	\end{aligned}
\end{equation}

Thirdly, for the $F_{\sigma}$--subset $\bigcup\limits_{l=1}^{\infty}\Big(T_{l} \cap f^{-1}\mathcal{O} _{n}\Big)$ of $f^{-1}\mathcal{O}_{n}$ and its neighborhood $\bigcup\limits_{l=1}^{\infty} U_{n}^{ 2^{n}-1} (l) $ there are a neighborhood $\mathcal{O}_{n+1}^{2^{n}-1} \subset \mathcal{O}_{n}$ of $y$ and open in $f^{-1} \mathcal{O}_{n+1}^{2^{n}-1} $ subsets $\widetilde{V}^{2^{n} -1}(l)$, $l \in \mathbb{N}$, such that
\begin{equation} \label{eqno(2.6.1.3)}
     \begin{aligned}
T_{l} \cap f^{-1}\mathcal{O}_{n+1}^{2^{n}-1}\ \ \subset \widetilde{V}_{l}^{2^ {n}-1} & \subset\ \ \mathrm{cl}_{\mathcal{O}_{n+1}^{2^{n}-1}} \widetilde{V}_{l}^ {2^{n}-1} \subset \\
         & \subset \bigcup\limits_{l=1}^{\infty}\Big( U_{n}^{2^{n}-1}(l) \cap f^{-1}\mathcal{O} _{n+1}^{2^{n}-1} \Big), \ l \in \mathbb{N}.
     \end{aligned}
\end{equation}
Put $\mathcal{O}_{n+1} = \bigcap\limits_{p=0}^{2^{n}-1} \mathcal{O}_{n}^{p}$ , $V_{l}^{p} = \widetilde{V}_{l}^{p} \cap f^{-1} \mathcal{O}_{n+1}$, $l \in \mathbb{N}$, $p\in\overline{0, 2^{n}-1}$. Then for the families of sets $\{V_{l}^{p} \}_{p=0}^ {2^{n}-1}$, $l \in \mathbb{N}$, the following hold. For any $l \in \mathbb{N}$

\begin{itemize}
	\item[(i)] directly from the construction it follows that $V_{l}^{p} \supset V_{l}^{p+1}$, $p\in\overline{0,2^{n}-2 }$;
	
	\item[(ii)] from the inclusions (\ref{eqno(2.6.1.1)}) and the fulfillment of condition (b) for regular $2^{n}$-partitions $\{ U^{k}_ {n}(l)\}_{k=0}^{2^{n}-1}$ of the subspace $f^{-1}\mathcal{O}_{n}$ it follows that
	\begin{equation*}
	\begin{aligned}
	F \cap \mathrm{cl}_{f^{-1}\mathcal{O}_{n+1}} V_{l}^{0}=\varnothing,\\
	\mathrm{cl}_{ f^{-1}\mathcal{O}_{n+1}}\Big(\bigcup\limits_{k=1}^{2^{n}-1} (U_{ n}^{k}(l)\cap f^{-1}\mathcal{O}_{n+1})\Big)\subset V_{l}^{0};
	\end{aligned}
	\end{equation*}
	
	\item[(iii)] from the inclusions (\ref{eqno(2.6.1.2)}) it follows that for $p\in\overline{0, 2^{n}-3}$
	\begin{equation*}
	\begin{aligned}
	(\bigcup\limits_{l=1}^{\infty}\bigcup\limits_{k=0}^{p} (U_{n}^{k}(l)\cap f^{-1}\mathcal {O}_{n+1})\big)\cap\mathrm{cl}_{ f^{-1}\mathcal{O}_{n+1}} V_{l}^{p+1} =\varnothing,\\
	\mathrm{cl}_{ f^{-1}\mathcal{O}_{n+1}}(\bigcup\limits_{k=p+2}^{2^{n}-1} (U_{ n}^{k}(l)\cap f^{-1}\mathcal{O}_{n+1})\big)\subset V_{l}^{p+1};
	\end{aligned}
	\end{equation*}
	
	\item[(iv)] from the inclusions (\ref{eqno(2.6.1.3)}) and the fulfillment of condition (c) for a regular $2^{n}$-partitions $\{ U^{k}_ {n}(l) \}_{k=0}^{2^{n}-1}$ of the subset $f^{-1}\mathcal{O}_{n}$ it follows that
	\begin{equation*}
	\begin{aligned}
	(\bigcup\limits_{l=1}^{\infty}\bigcup\limits_{k=0}^{2^{n}-2} (U_{n}^{k}(l)\cap f^ {-1}\mathcal{O}_{n+1})\big)\cap\mathrm{cl}_{ f^{-1}\mathcal{O}_{n+1}} V_{l} ^{2^{n}-1} = \varnothing,\\
	T_{l}\cap f^{-1}\mathcal{O}_{n+1}\subset V_{l}^{2^{n}-1}.
	\end{aligned}
	\end{equation*}
\end{itemize}

Put $U_{n+1}^{2k} (l) = (U_{n}^{k}(l) \setminus V_{l}^{k})\cap f^{-1}\mathcal{ O}_{n+1}$, $ U_{n+1}^{2k+1}(l) = (U_{n}^{k}(l)\cap V_{l}^{k}) \cap f^{-1}\mathcal{O}_{n+1}$, $ k\in\overline{0, 2^{n}-1}$, $l \in \mathbb{N}$. Then from (ii) -- (iv) it follows that $\{U^{k}_{n+1}(l) |\ k\in\overline{0 , 2^{n+1}-1}\}$ is a regular $2^{n+1}$--partition of the subspace $f^{-1} \mathcal{O}_{n+1}$, $l \in \mathbb{N}$, for which condition of the Definition~\ref{defCoherenceSetBinPart}  is fulfilled (i.e. $U_{n+1}^{2k}(l)\cup U_{n+1}^{2k+1}(l)=U_{ n}^{k}(l)\cap f^{-1}\mathcal{O}_{n+1},\ k\in\overline{0, 2^{n}-1}$).

The constructed sequence of neighborhoods $\mathcal{O}_{n}$ and the families of regular $2^{n}$-partitions of the subspace $f^{-1} \mathcal{O}_{n}$, $n\in \omega_0$, are the consistent families $\{\mathcal{O}_{n}, \{U_{n}^{k}(l)| \ k\in\overline{0, 2^{n}-1}\}|\ n\in\mathbb\omega_0\}$ of binary partitions of the mapping $f: X \to Y$ at $y$ for each $l \in \mathbb{N}$. The fulfillment of properties (b) and (c) follows from (ii) and (iv).
\end{proof}

\begin{theorem}\label{thFuncCharSigmNorm}
	For a mapping $f: X \to Y$ the following conditions are equivalent.
	\begin{itemize}
	\item[{\rm (A)}] The mapping $f$ is $\sigma$--normal.
	\item[{\rm (B)}] For any $\mathcal{O}\in\tau_{Y}$, any disjoint closed submapping $f_{\mathcal{O}}|_{F}: F \to \mathcal{O}$ and $F_{\sigma}$-submapping $f_{\mathcal{O}}|_{T}: T\to Y$ of the mapping $f_{\mathcal{O}}: f^{ -1}\mathcal{O} \to \mathcal{O}$ and any point $y \in \mathcal{O}$
	there exist consistent families of binary partitions $\{\mathcal{O}_{n}, \{U_{n}^{k}(l)|\ k\in\overline{0, 2^{n}-1}\} |\ n\in\omega_0\}$, $l \in \mathbb{N}$, of the mapping $f: X \to Y$ at $y$ such that for any $n \in \mathbb{N} $
	\begin{itemize}
	\item[{\rm (a)}] $T \cap f^{-1}\mathcal{O}_{1} = \bigcup\limits_{l=1}^{\infty} T_{l}$ , where $T_{l}$ is closed in $f^{-1}\mathcal{O}_{1}$, $l \in \mathbb{N}$\rm;
	\item[{\rm (b)}] $F\cap f^{-1}\mathcal{O}_n\subset U_{n}^{0}(l)$,\ \ $T_{l} \cap f^{-1}\mathcal{O}_n\subset U_{n}^{2^{n}-1}(l)$, $l \in \mathbb{N}$\rm;
	\item[{\rm (c)}] $F\cap \mathrm{cl}_{ f^{-1}\mathcal{O}_n}(\bigcup\limits_{k=1}^{2^{ n}-1} U_{n}^{k}(l))=\varnothing,\ \ \mathrm{cl}_{ f^{-1}\mathcal{O}_n}(\bigcup\limits_{k =0}^{2^{n}-2} U_{n}^{k}(l)) \cap T_{l} = \varnothing$, $l \in \mathbb{N}$.
	\end{itemize}
	\item[{\rm (C)}] For any $\mathcal{O}\in\tau_{Y}$ and any disjoint closed submapping $f_{\mathcal{O}}|_{F}: F \to Y$ and $F_{\sigma}$--submapping $f_{\mathcal{O}}|_{T}: T\to Y$ of the mapping $f_{\mathcal{O}}: f^{-1} \mathcal{O} \to \mathcal{O}$ and any point $y \in \mathcal{O}$ there exist a  neighborhood $\mathcal{O}y$ of $y$ and a $f$-equicontinuous  at $y$ family of functions $\varphi_{l}: X \to [0,1]$, $l \in \mathbb{N}$, such that
	\begin{itemize}
	\item[{\rm (a)}] $T\cap f^{-1}\mathcal{O}y=\bigcup\limits_{n=1}^\infty T_{l}$, $T_{l }$ are closed in $f^{-1}\mathcal{O}y$, $l\in\mathbb{N}$;
	\item[{\rm (b)}] $\mathrm{osc}_{\varphi_{l}}(f^{-1}\mathcal{O}y) < \frac{1}{2}$, $l \in \mathbb{N}$;
	\item[{\rm (c)}] $F\cap f^{-1}\mathcal{O}y\subset\varphi_{l}^{-1}(0)\cap f^{-1} \mathcal{O}y$ and $T_{l}\cap f^{-1}\mathcal{O}y \subset \varphi_{l}^{-1}(1)\cap f^{-1} \mathcal{O}y$, $l \in \mathbb{N}$;
	\item[{\rm (d)}] $T_{l}\cap f^{-1}\mathcal{O}y \subset \mathrm{int}_{ f^{-1}\mathcal{O} y}(\varphi_{l}^{-1}(\frac{1}{2}, 1]\cap f^{-1}\mathcal{O}y)$, \\
	$ \mathrm{cl}_{ f^{-1}\mathcal{O}y}(\varphi_{l}^{-1}(\frac{1}{2}, 1]\cap f^{- 1}\mathcal{O}y)\cap F = \varnothing$, $l \in \mathbb{N}$.
	\end{itemize}
	\end{itemize}
\end{theorem}

\begin{proof}
\textbf{(A) $\Rightarrow$ (B)} by Lemma \ref{lemFSetPartitionOfSigmaNormalMap}.

\textbf{(B) $\Rightarrow$ (C).}
Condition (B) implies the existence of the consistent families of binary partitions $\{ \mathcal{O}_{n}, \{U^{k}_{n}(l) | \ k\in\overline{0,2^{n}-1} \}|\ n \in \omega_{0} \}$, $l \in \mathbb{N}$, of the mapping $f:X \to Y$ at $y$ for which satisfy conditions (a) --- (c) of (B).

By Proposition~\ref{prFContMapOnCorPart} for each $l \in \mathbb{N}$ by the consistent family of binary partitions $\{ \mathcal{O}_{n}, \{U^{k}_{n}( l) | \ k\in\overline{0,2^{n}-1} \}|\ n \in \omega_{0} \}$, $l \in \mathbb{N}$, of the mapping $f: X \to Y $ at $y$,
a $f$-continuous at  $y$ functions $\varphi_{n}^{l}: X \to [0,1]$ 
\begin{equation*}
\varphi_{n}^{l} (x) = \left\{ \begin{array}{ll}
0, & x \in X\setminus f^{-1} \mathcal{O}_{n}, \\
\frac{k}{2^{n}-1}, & x \in U^{k}_{n}(l), \ k\in\overline{0, 2^{n}-1} \ \
\end{array}\right.
\end{equation*}
are constructed. And they yield the construction of the family of $f$-continuous at  $y$ functions $\varphi_{l} : X \to [0,1]$
\begin{equation*}
\varphi_{l} (x) = \left\{ \begin{array}{ll}
\varphi_{n}^{l}(x), & x \in f^{-1}\mathcal{O}_{n} \setminus f^{-1}\mathcal{O}_{n+1 },\ n\in\mathbb\omega_0,\\
\lim\limits_{n \to \infty}\varphi_{n}^{l}(x), & x\in\bigcap\limits_{n=0}^{\infty} f^{-1}\mathcal {O}_{n},\\
\end{array}\right.
\end{equation*}
such that ${\rm osc}_{\varphi_{l}}(f^{-1}\mathcal{O}_{n})\leqslant \frac{1}{2^{n}-1}$ for all $l\in \mathbb{N}, \ n \in \omega_{0}$. Therefore, the family of functions $\{\varphi_{l} \}$ is $f$-equicontinuous at $y$.

Moreover, firstly, for all $l \in \mathbb{N}$

$\varphi_{l} (x) = 0$ for $x \in F$ (since $F \subset U_{n}^{0}(l)$, $n\in \mathbb{N} $),

$\varphi_{l} (x) = 1$ for $x \in T_{l}$ (since $T_{l} \subset U_{n}^{2^{n}-1}(l )$, $ n \in \mathbb{N}$).

Secondly, from the $f$-continuity of $\varphi_{l}$  the existence of a neighborhood $\mathcal{O}y = \mathcal{O}_{2}$ of  $y$ such that $\mathrm{osc }_{\varphi_{l}}(f^{-1}\mathcal{O}y)<\frac{1}{2}$, $l \in \mathbb{N}$, follows.
Then by Lemma \ref{lemOscDisj}
\begin{equation*}
F \cap f^{-1}\mathcal{O}y\subset \varphi_{l}^{-1}(0)\cap f^{-1}\mathcal{O}y\subset f^{- 1}\mathcal{O}y \setminus \mathrm{cl}_{f^{-1}\mathcal{O}y}(\varphi_{l}^{-1}[\tfrac{1}{2} , 1]\cap f^{-1}\mathcal{O}y),
\end{equation*}
and therefore, $T_{l}\cap f^{-1}\mathcal{O}y \subset \mathrm{int}_{ f^{-1}\mathcal{O}y}(\varphi_{l }^{-1}(\frac{1}{2}, 1]\cap f^{-1}\mathcal{O}y)$, $l \in \mathbb{N}$.
So $T_{l}\cap f^{-1}\mathcal{O}y\subset U^{1}_{1}(l) \cap f^{-1}\mathcal{O}y\subset \mathrm{int}_{ f^{-1}\mathcal{O}y} (\varphi_{l}^{-1}[\frac{1}{2}, 1]\cap f^{-1 }\mathcal{O}y), \l \in \mathbb{N}$.

\textbf{(C) $\Rightarrow$ (A).} Consider an arbitrary disjoint closed subset $F$ and $F_{\sigma}$--subset $T = \bigcup\limits_{l =1}^{\infty } T_{l}$ of the set $f^{-1}\mathcal{O} $ for an arbitrary mapping $f_{\mathcal{O}}: f^{-1}\mathcal{O} \to \mathcal{O }$. Let $y \in \mathcal{O}$ be arbitrary.
Then, it follows from item (d) of condition (C) that the neighborhood $\mathcal{O}y$ of the point $y$ and the neighborhoods $\mathrm{int}_{f^{-1}\mathcal{O}y}( \varphi_{l}^{-1}(\frac{1}{2}, 1]\cap f^{-1}\mathcal{O}y)$ of the sets $T_{l}\cap f^{-1}\mathcal{O}y$, $l\in\mathbb N$, are required.
\end{proof}


\section{Perfect normality of a mapping.}

\begin{definition} \label{defCommonPerfNorm}
A mapping $f:X\to Y$ is {\it perfectly normal} if for any open subset $O\subset X$ and for any point $y\in Y$ there exist 
a neighborhood $\mathcal Oy$ of $y$ and
a $f$-equicontinuous at $y\in Y$  countable family of functions $\{\varphi_{l}: X\to [0, 1]\ |\ l\in\mathbb N\}$ such that
\begin{itemize}
\item[{\rm (1)}] $O\cap f^{-1} \mathcal{O}y=\bigcup\limits_{l\in\mathbb{N}} ({\varphi_{l}} ^{-1} (1)\cap f^{-1} \mathcal{O}y),$
\item[{\rm (2)}] $f^{-1} \mathcal{O}y\setminus O\subset{\varphi_{l}}^{-1}(0)\cap f^{- 1} \mathcal{O}y,\ l\in\mathbb{N}.$
\end{itemize}
\end{definition}

\begin{remark} \label{remOnDefCommonPerfNorm} In the Definition~{\rm\ref{defCommonPerfNorm}} one has. 
	\begin{itemize}
	\item[{\rm (a)}] One can assume that $O \cap f^{-1} \mathcal{O}y = \bigcup\limits_{l \in \mathbb{N}} {\rm cl }_{f^{-1} \mathcal{O}y} (\varphi_{l}^{-1}(1)\cap f^{-1} \mathcal{O}y)$.
	\item[{\rm (b)}] The condition {\rm (2)} is equivalent to the following one $f^{-1} \mathcal{O}y\setminus O=\bigcap\limits_{l\in\mathbb{N} } {\varphi_{l}}^{-1} (0)\cap f^{-1} \mathcal{O}y$.
	\item[{\rm (c)}] Any open submapping is a $F_{\sigma}$-submapping.
	\end{itemize}
\end{remark}

\begin{proposition}\label{prHerFuncPerfNormMap}
	A submapping $f|_{X_{0}}: X_{0} \to Y$ of a perfectly normal mapping $f: X \to Y$ is perfectly normal.
\end{proposition}

\begin{proof}
	Consider an arbitrary submapping $f|_{X_{0}}: X_{0} \to Y$ of the mapping $f: X \to Y$, a point $y \in Y$ and an open in $X_{0}$ subset  $O$. There is an open in $X$ subset $\widetilde{O}$ such that $O = \widetilde{O} \cap X_{0}$. From the perfect normality of the mapping $f$ it follows that there is a neighborhood $\mathcal{O}y$ of $y$ and a countable family of $f$-equicontinuous at $y\in Y$ functions $\{\varphi_{l}: X\to [0, 1]\ |\ l\in\mathbb N\}$, such that 
	$$\widetilde{O} \cap f^{-1} \mathcal{O}y = \bigcup\limits_{l\in\mathbb N} \varphi_{l}^{-1} (1)\cap f ^{-1} \mathcal{O}y ,\ f^{-1} \mathcal{O}y\setminus \widetilde{O}\subset \varphi_{l}^{-1}(0) \cap f ^{-1}\mathcal{O}y ,\ l \in\mathbb{N}.$$
	
	It is obvious that from the $f$-equicontinuity at $y\in Y$ of a countable family of functions $\{\varphi_{l}: X\to [0, 1]\ |\ l\in\mathbb N\}$ it follows $f|_{X_{0}}$--equicontinuity at $y\in Y$ of a countable family of functions $\{\varphi_{l}|_{X_{0}}: X_{0} \to [0, 1]\ |\ l\in\mathbb{N}\}$.
	
	Since $X_{0} \subset X$ and $ \mathcal{O}y \subset Y$, then $\widetilde{O} \cap X_{0} \cap f^{-1} \mathcal{O}y = O \cap (f|_{X_{0}})^{-1}\mathcal{O}y$ and $X_{0} \cap \varphi_{l}^{-1} M = (\varphi_ {l}|_{X_{0}})^{-1}M$, for $M \subset [0,1]$, $l \in \mathbb{N}$. Hence
	\begin{equation*}
	\begin{aligned}
	O \cap (f|_{X_{0}})^{-1} \mathcal{O}y & = \bigcup\limits_{l\in\mathbb{N}} (\varphi_{l}|_{ X_{0}})^{-1} (1) \cap f^{-1} \mathcal{O}y ,\\
	(f|_{X_{0}})^{-1} \mathcal{O}y\setminus O & \subset (\varphi_{l}|_{X_{0}})^{-1}(0 ) \cap f^{-1} \mathcal{O}y ,\ l \in\mathbb{N}.
	\end{aligned}
	\end{equation*}
\end{proof}
	
\begin{proposition}
	A perfectly normal mapping $f:X\to Y$ is prenormal.
\end{proposition}
\begin{proof}
	Consider arbitrary disjoint subsets $F$ and $T$ closed in $X$ and a point $y \in Y$.
	
	From the perfect normality of the mapping $f$, for an open in $X$ subset $U = X \setminus F$ there exist a neighborhood $\mathcal{O}^{\prime}y $ of $y$ and a $f$-equicontinuous at $y$ family of functions $\{ {\varphi}_{l}: X \to [0,1] \}_{l \in \mathbb{N}}$ such that
	$$U \cap f^{-1} \mathcal{O}^{\prime}y = \bigcup\limits_{l\in\mathbb{N}} {\varphi}_{l}^{-1} (1) \cap f^{-1} \mathcal{O}y,\ F = f^{-1} \mathcal{O}^{\prime}y \setminus U \subset{\varphi}_{l }^{-1}(0) \cap f^{-1} \mathcal{O}y \ l \in\mathbb{N}.$$
	
	Similarly, for an open in $X$ subset $V = X \setminus T$ there are a neighborhood $\mathcal{O}^{\prime\prime}y $ of $y$ and a $f$-equicontinuous at $y$ family of functions $\{ {\psi}_{l}: X \to [0,1] \}_{l \in \mathbb{N}}$ such that
	$$V \cap f^{-1} \mathcal{O}^{\prime\prime}y = \bigcup\limits_{l\in\mathbb{N}} {\psi}_{l}^{- 1} (1) \cap f^{-1} \mathcal{O}y, \ T = f^{-1} \mathcal{O}^{\prime\prime}y \setminus V \subset{\psi }_{l}^{-1}(0) \cap f^{-1} \mathcal{O}y,\ l \in\mathbb{N}.$$
	
	Since the families of functions  $\{\varphi_{l}\}_{l \in \mathbb{N}}$ and $\{\psi_{l}\}_{l \in \mathbb{N}}$ are $f$-equicontinuous at $y$, there is a neighborhood $\mathcal{O}y \subset \mathcal{O}^{\prime}y \cap \mathcal{O}^{\prime\prime}y$ of $y$ such that
	${\rm osc}_{\varphi_{l}}(f^{-1}\mathcal{O}y) < \frac{1}{2}$ and ${\rm osc}_{\psi_{ l}}(f^{-1}\mathcal{O}y) < \frac{1}{2}$, $l \in \mathbb{N}$. Besides,
	\begin{equation}\label{eqPerfNormIsNormNew01}
		\begin{aligned}
			F \cap f^{-1}\mathcal{O}y \subset \bigcup\limits_{l\in\mathbb{N}}{\rm int}_{f^{-1} \mathcal{O}y }(\psi_{l}^{-1}(\tfrac{1}{2}, 1]\cap f^{-1} \mathcal{O}y) \subset \bigcup\limits_{l\in \mathbb{N}}{\rm cl}_{f^{-1} \mathcal{O}y}(\psi_{l}^{-1}(\tfrac{1}{2}, 1]\cap f^{-1} \mathcal{O}y), \\
			T \cap f^{-1}\mathcal{O}y \subset \bigcup\limits_{l\in\mathbb{N}}{\rm int}_{f^{-1} \mathcal{O}y }(\varphi_{l}^{-1}(\tfrac{1}{2}, 1] \cap f^{-1}\mathcal{O}y) \subset \bigcup\limits_{l\in \mathbb{N}}{\rm cl}_{f^{-1} \mathcal{O}y}(\varphi_{l}^{-1}(\tfrac{1}{2}, 1] \cap f^{-1}\mathcal{O}y).
		\end{aligned}
	\end{equation}

	From Lemma \ref{lemOscDisj} and inclusions (\ref{eqPerfNormIsNormNew01}) we have
	\begin{equation*}
	\begin{aligned}
		F \cap f^{-1}\mathcal{O}y \cap \bigcup\limits_{l\in\mathbb{N}}{\rm cl}_{f^{-1} \mathcal{O}y }(\varphi_{l}^{-1}(\tfrac{1}{2}, 1]\cap f^{-1} \mathcal{O}y) & = \varnothing, \\
		T \cap f^{-1}\mathcal{O}y \cap \bigcup\limits_{l\in\mathbb{N}}{\rm cl}_{f^{-1} \mathcal{O}y }(\psi_{l}^{-1}(\tfrac{1}{2}, 1]\cap f^{-1} \mathcal{O}y) & = \varnothing.
	\end{aligned}
	\end{equation*}
	
	The disjoint closed in $f^{-1}\mathcal{O}y$ subsets
	$F \cap f^{-1}\mathcal{O}y$ and
	$T \cap f^{-1}\mathcal{O}y$, countable families of neighborhoods $\{{\rm int}_{f^{-1} \mathcal{O}y}(\varphi_{l}^{-1}(\tfrac{1}{2}, 1] \cap f^{-1} \mathcal{O}y)\}_{l \in \mathbb{N}}$ and
	$\{{\rm int}_{f^{-1} \mathcal{O}y}(\psi_{l}^{-1}(\tfrac{1}{2}, 1] \cap f^ {-1} \mathcal{O}y)\}_{l \in \mathbb{N}}$ satisfy the normalizing Lemma \cite[Chapter 1, $\S$~5, Lemma~2]{AlexandrovPasynkov}.
	Thus, the subsets $F$ and $T$ are disjoint in $f^{-1}\mathcal{O}y$. This implies that the mapping $f$ is prenormal.
\end{proof}
	
Since the perfect normality of a mapping is a hereditary property, and any perfectly normal mapping is prenormal, then by \cite[Proposition~7]{liseev2} we have.

\begin{theorem}\label{thCommonPerfNormMapIsHerNorm}
A perfectly normal mapping $f:X\to Y$ is hereditarily normal.
\end{theorem}

The following definition is a variant of extending the concept of a functionally open set to the case of mappings. It differs from the definition at~\cite[p. 76]{PasOtobrFunct} and allows to define the necessary condition of perfect normality, similar to Vedenisov’s condition of perfect normality for spaces.

\begin{definition} \label{defFuncOpenMap}
	A submapping $f|_{U}: U \to Y$ of a mapping $f: X \to Y$ is  {\it $f$--functionally open} if for any point $y$ there are a neighborhood $\mathcal{O}y$ of $y$ and a $f$-continuous function $\varphi: X\to [0,1]$ such that $U \cap f^{-1}\mathcal{O}y= \varphi^{ -1}(0,1] \cap f^{-1} \mathcal{O}y$.
	
	A submapping $f|_{F}: F \to Y$ of a mapping $f: X \to Y$ is {\it $f$--functionally closed} if for any point $y$ there are a neighborhood $\mathcal{O}y$ of $y$  and a $f$-continuous function $\varphi: X\to [0,1]$ such that $F\cap f^{-1}\mathcal{O}y = \varphi^{ -1}(0) \cap f^{-1} \mathcal{O}y$.
\end{definition}

\begin{remark} \label{aaa}
	A submapping $f|_{U}: U \to Y$ of a mapping $f: X \to Y$ is $f$-functionally open iff the submapping $f|_{X\setminus U} : X\setminus U \to Y$ is $f$-functionally closed.
\end{remark}
	
\begin{proposition}
	Any open submapping $f|_{O}: O \to Y$ {\rm (}closed submapping $f|_{F}: F \to Y${\rm )} of a perfectly normal mapping $f: X\to Y$ is $f$-functionally open {\rm (}respectively, $f$-functionally closed{\rm )}.
\end{proposition}

\begin{proof} Consider an arbitrary point $y \in Y$ and an open in $X$ subset $O$.
    From the perfect normality of the mapping $f$ it follows that there are a $f$-equicontinuous countable family of functions $\{\varphi_{l}: X\to [0, 1]\ |\ l\in\mathbb{N} \}$ at $y\in Y$ and a neighborhood $\mathcal{O}y$ of $y$ such that
    $$O\cap f^{-1} \mathcal{O}y=\bigcup\limits_{l\in\mathbb{N}} \varphi_{l}^{-1} (1) \cap f^{ -1} \mathcal{O}y,\ f^{-1} \mathcal{O}y\setminus O\subset\varphi_{l}^{-1}(0) \cap f^{-1} \mathcal{O}y,\ l\in\mathbb{N}.$$
   
The series $\sum\limits_{l=1}^{\infty}\big|\big|\frac{\varphi_{l}}{2^{l}}\big|\big|_{f^{- 1}\mathcal{O}y}$ converges. So, by Lemma \ref{lemSumFcontMapFcont} the function $\varphi: X \to [0,1]$
    \begin{equation*}
\varphi (x)=\sum_{l=1}^{\infty}\frac{\varphi_{l}(x)}{2^{l}}
    \end{equation*}
    is $f$-continuous at $y$. If $x \in O \cap f^{-1}\mathcal{O}$, then there exists $l \in \mathbb{N}$ such that $x \in\varphi^{-1}(1) $. Then $\varphi (x) \geqslant \frac{\varphi_{l}(x)}{2^{l}} = \frac{1}{2^{l}} > 0 $. Therefore, $O \cap f^{-1}\mathcal{O}y = \varphi^{-1}(0, 1]\cap f^{-1} \mathcal{O}y$ and the submapping is $f$--functionally open.

The proof of the $f$-functional closeness of a closed submapping follows from the Remark~\ref{aaa}.
\end{proof}

\begin{definition}\label{defSigmPerfNorm}
	A normal mapping $f: X \to Y$ is called a {\it $co$-perfectly normal mapping} if its any open submapping $f|_U: U\to Y$, $U\in\tau_{X}$ is a $F_{\sigma}$-submapping~{\rm\cite[Definition 8]{liseev3}}.
	
	A $\sigma$-normal mapping $f: X \rightarrow Y$ is called a $co$-$\sigma$-{\it perfectly normal mapping} if its any open submapping $f|_U: U\to Y$, $U\in\tau_{X}$ is a $F_{\sigma}$-submapping~{\rm\cite[Definition 9]{liseev3}}.
\end{definition}

Any $co$-$\sigma$-perfectly normal mapping is hereditarily normal \cite[Theorem 14]{liseev3}. Any submapping of $co$-$\sigma$-perfectly normal mapping is $co$-perfectly normal \cite[Corollary 15]{liseev3}.

\begin{theorem}[{\it a functional characterization of a $co$-$\sigma$-perfect normality of a mapping}]\label{thFuncCharPerfSigmNormmap}
	A mapping $f: X \to Y$ is $co$-$\sigma$-perfectly normal iff for any $\mathcal{O}\in \tau_{Y}$, for any open subset $U\subset f^{-1}\mathcal{O}$, and for any $F_{\sigma}$-subset $F = \bigcup\limits_{l\in\mathbb{N}} F_{l }\subset U$, where the sets $F_{l}$ are closed in $f^{-1}\mathcal{O}$, $l \in \mathbb{N}$, for any point $y\in \mathcal {O}$
	there are a neighborhood $\mathcal{O}y$ of $y$ and a $f$-equicontinuous at $y$ family of functions $\varphi_{l}: X\to [0,1]$ and $\psi_{ l}: U \to [0,1]$, $l \in \mathbb{N}$, such that
	\begin{itemize}
	\item[{\rm (a)}] $\mathrm{osc}_{\varphi_{l}} ( f^{-1}\mathcal{O}y) < \frac{1}{2}$, $\mathrm{osc}_{\psi_{l}} ( f^{-1}\mathcal{O} y) < \frac{1}{2}$, $l \in \mathbb{N}$;
	\item[{\rm (b)}] $f^{-1}\mathcal{O}y \setminus U \subset \varphi_{l}^{-1}(0)\cap f^{-1} \mathcal{O}y$, $F_{l}\cap f^{-1}\mathcal{O}y\subset\varphi_{l}^{-1}(1)\cap f^{-1} \mathcal{O}$, $l\in\mathbb{N}$;
	\item[{\rm (c)}] $f^{-1}\mathcal{O}y\setminus U \subset \psi_{l}^{-1}(0)\cap f^{-1} \mathcal{O}$, $U\cap f^{-1}\mathcal{O}y = \bigcup \limits_{l\in\mathbb{N}} \psi^{-1}_{l}( 1)\cap f^{-1}\mathcal{O}$, $l \in \mathbb{N}$.
	\end{itemize}
\end{theorem}

\begin{proof} {\it Necessity.} Without loss of generality, we will assume that $Y = \mathcal{O}$. Let fix an arbitrary point $y \in Y$. From the $co$-$\sigma$-perfect normality of the mapping $f$ it follows that its open submapping $f|_{U}: U \to Y$ has type of $F_{\sigma}$.
It means that there is a neighborhood $\mathcal{O}'y$ of $y$ such that $f^{-1}\mathcal{O}'y\cap U = \bigcup\limits_{l=1}^{\infty} T'_i$, where $T'_{l}$ is closed in $f^{-1}\mathcal{O}'y$ for all $l \in \mathbb{N}$.
	
	For a countable family of closed subsets $T'_{l}$, $F'_{l} = f^{-1}\mathcal{O}'y \cap F_{l}$, $l \in \mathbb{N}$, we have
	\begin{equation*}
	\bigcup\limits_{l=1}^{\infty} (T'_{l} \cup F'_{l}) = f^{-1}\mathcal{O}'y\cap U \Longrightarrow \bigcup\limits_{l=1}^{\infty} (T'_{l} \cup F'_{l}) \cap \big(f^{-1}\mathcal{O}'y \setminus U \big) = \varnothing.
	\end{equation*}
	
	The set $f^{-1}\mathcal{O}'y \setminus U$ is closed in $f^{-1}\mathcal{O}'y$. The mapping $f: X \to Y$ is $\sigma$-normal. By Theorem~\ref{thFuncCharSigmNorm} there exist a neighborhood $\mathcal{O}y \subset \mathcal{O}'y$ of $y$ and a $f$-equicontinuous at $y$ family of functions $\varphi_{l}: X \to [0,1]$, $l \in \mathbb{N}$, $\psi_{l}: X \to [0,1]$, $l \in \mathbb {N}$, such that
	\begin{itemize}
	\item[{\rm (a$^\prime$)}] $\mathrm{osc}_{\varphi_{l}}(f^{-1}\mathcal{O}y) < \frac{1} {2}$, $\mathrm{osc}_{\psi_{l}}(f^{-1}\mathcal{O}y) < \frac{1}{2}$, $l \in \mathbb {N}$;
		
	\item[{\rm (b$^\prime$)}] $U\cap f^{-1}\mathcal{O}y = \bigcup\limits_{l=1}^{\infty} T_{l }$, where the sets $T_{l}=T'_{l}\cap f^{-1}\mathcal{O}y$ are closed in $f^{-1}\mathcal{O}y$, and the sets $F_{l}\cap f^{-1}\mathcal{O}y\subset U\cap f^{-1}\mathcal{ O}y$, $l \in \mathbb{N}$, are closed in $f^{-1}\mathcal{O}y$;
		
	\item[{\rm (c$^\prime$)}] $f^{-1}\mathcal{O}y \setminus U \subset \varphi_{l}^{-1}(0)\cap f ^{-1}\mathcal{O}y$, $f^{-1}\mathcal{O}y \setminus U \subset \psi_{l}^{-1}(0)\cap f^{- 1}\mathcal{O}y$,
	 
	$F_{l}\cap f^{-1}\mathcal{O}y \subset \varphi_{l}^{-1}(1)\cap f^{-1}\mathcal{O}y$, $T_{l} \cap f^{-1}\mathcal{O}y \subset \psi_{l}^{-1}(1)\cap f^{-1}\mathcal{O}y$, $l\in\mathbb{N}$.
	  \end{itemize}
From the conditions (a$^\prime$) --- (c$^\prime$)  the conditions of the theorem follows.
 
\medskip
 
{\it Sufficiency.} Let us show that any open submapping of the mapping $f: X\to Y$ is of type $F_{\sigma}$. Without loss of generality, we will assume that $Y = \mathcal{O}$. The set $U$ is open in $X$.

According to conditions of the theorem, for any point $y\in Y$ there exist its neighborhood $\mathcal{O}y$ and a $f$-equicontinuous at $y$ family of functions $\psi_{l}: f^{-1 }\mathcal{O}y \to [0,1]$, $l\in \mathbb{N}$, such that:
\begin{equation*}
\begin{aligned}
\mathrm{osc}_{\psi_{l}} ( f^{-1}\mathcal{O}y) < \tfrac{1}{2}, \ f^{-1}\mathcal{O}y \setminus U \subset \psi_{l}^{-1}(0)\cap f^{-1}\mathcal{O}y, \ l \in \mathbb{N},\text{ and } \\
U \cap f^{-1}\mathcal{O}y=\bigcup\limits_{l \in \mathbb{N}} \psi^{-1}_{l}(1)\cap f^{- 1}\mathcal{O}y.
\end{aligned}
\end{equation*}
 
By Lemma~\ref{lemOscDisj} (for $a = 0$, $b = 1$) we have $\{ x \in f^{-1} \mathcal{O}y \; | \; \psi_{l}(x) = 0 \} \cap \mathrm{cl}_{f^{-1} \mathcal{O}y} \{ x \in f^{-1} \mathcal{O} y\; | \; \psi_{l}(x) = 1\} = \varnothing$. Let $T_{l} = \mathrm{cl}_{f^{-1} \mathcal{O}y} \{ x \in f^{-1} \mathcal{O}y \; | \; \psi_{l}(x) = 1\}$, $l\in\mathbb{N}$. Then, $U\cap f^{-1}\mathcal{O}y=\bigcup\limits_{l\in\mathbb N} T_{l}$ and the mapping $f$ is of type $F_{\sigma}$.
 
Let's show that the mapping $f: X\to Y$ is $\sigma$-normal. Without loss of generality, we will assume that $Y = \mathcal{O}$. Let $F = \bigcup\limits_{l \in \mathbb{N}} F_{l}$ be a $F_{\sigma}$-subset of $X$, $F_{l}$ are closed in $X$, $T$ is closed in $X$ and $T \cap F = \varnothing$. Then
$F \subset U= X \setminus T$.
 
From the conditions of the theorem it follows that for any point $y \in Y$ there exist its neighborhood $\mathcal{O}y$ and a $f$-equicontinuous at $y$ family of functions $\varphi_{l}: X \to [0 ,1]$, $l \in \mathbb{N}$, such that
\begin{equation*}
\begin{aligned}
\mathrm{osc}_{\varphi_{l}} ( f^{-1}\mathcal{O}y) < \tfrac{1}{2}, \ F_{l} \cap f^{-1} \mathcal{O}y \subset \varphi_{l}^{-1}(1) \cap f^{-1}\mathcal{O}y, \\
f^{-1}\mathcal{O}y \setminus U \subset \varphi_{l}^{-1}(0) \cap f^{-1}\mathcal{O}y, \ l \in \mathbb{N}.
\end{aligned}
\end{equation*}

By Theorem~\ref{thFuncCharSigmNorm} the mapping $f: X \to Y$ is $\sigma$-normal.
\end{proof}

\begin{corollary}
	For a $\sigma$-normal mapping $f: X\to Y$ the following conditions are equivalent.
	\begin{itemize}
	\item[{\rm (A)}] The mapping $f: X \to Y$ is $co$-$\sigma$-perfectly normal.
	\item[{\rm (B)}] For any open submapping $f|_{U}: U \to Y$ of $f$ and any point $y\in Y$ there are a neighborhood $\mathcal{O}y$ of $y$ and a $f$-equicontinuous at $y$ family of functions $\psi_{l}: X \to [0,1]$, $l \in \mathbb{N }$, such that
	\begin{itemize}
	\item[{\rm (1)}] $\mathrm{osc}_{\psi_{l}} ( f^{-1}\mathcal{O}y) < \frac{1}{2}$, $l \in \mathbb{N}$;
	\item[{\rm (2)}] $f^{-1} \mathcal{O}y \setminus U\subset \psi_{l}^{-1}(0)\cap f^{-1} \mathcal{O}y$, $l \in \mathbb{N}$, $U\cap f^{-1}\mathcal{O}y = \bigcup\limits_{l \in \mathbb{N}} \psi^{-1}_{l}(1)\cap f^{-1} \mathcal{O}y$.
	\end{itemize}
	\end{itemize}
\end{corollary}

From Theorem~\ref{thFuncCharPerfSigmNormmap} and point (c) of Remark~\ref{remOnDefCommonPerfNorm} we have.

\begin{proposition}\label{prHerFuncPerfNormMapbb}
Any $co$-$\sigma$-perfectly normal mapping is perfectly normal. Any perfectly normal mapping is $co$-perfectly normal.
\end{proposition}

\begin{remark}
In the case of a constant mapping, $co$-perfect normality, perfect normality and $co$-$\sigma$-perfect normality of the mapping coincide.
\end{remark}

\end{document}